\documentclass[12pt,a4paper]{amsart}
\usepackage{mathtools,amsmath,amsfonts,amssymb,amsthm,amsopn,array,MnSymbol,centernot,stmaryrd}
 
\usepackage{biblatex}

\usepackage{amsaddr}
\usepackage{etoolbox}
\patchcmd{\subsection}{-.5em}{.5em}{}{}

\newtheorem{definition}{Definition}[section]

\newtheorem{lemma}[definition]{Lemma}
\newtheorem{construction}[definition]{Construction}
\newtheorem{proposition}[definition]{Proposition}
\newtheorem{theorem}[definition]{Theorem}
\newtheorem{corollary}[definition]{Corollary}

\newtheorem*{claim*}{Claim}

\theoremstyle{remark}
\newtheorem{rmk}[definition]{Remark}
\theoremstyle{remark}

\theoremstyle{remark}

\theoremstyle{remark}

\theoremstyle{remark}

\theoremstyle{remark}

\theoremstyle{remark}

\theoremstyle{remark}

\theoremstyle{remark}

\theoremstyle{definition}

\newcommand{\N}{\mathbb{N}} 
\newcommand{\Q}{\mathbb{Q}}
\newcommand{\Z}{\mathbb{Z}}

\newcommand{\lex}{\operatorname{<_\mathrm{lex}}}
\newcommand{\LOfp}{\mathrm{LO}_{\mathrm{fp}}}
\newcommand{\dLOfpb}{\mathrm{dLO}_{\mathrm{fp}}^{\mathrm{b}}}
\newcommand{\ST}{\mathbf{ST}}
\newcommand{\ATST}{\mathbf{A}3\mathbf{ST}_\omega}
\newcommand{\ILO}{\mathbf{ILO}}
\newcommand{\AST}{\mathbf{AST}_\omega}
\newcommand{\dILO}{\mathbf{dILO}}
\newcommand{\dLO}{\mathbf{dLO}}
\newcommand{\lin}[2]{\mathrm{lin}(#1,#2)}
\newcommand{\LIN}[2]{\mathrm{LIN}(#1,#2)}
\newcommand{\hgt}[1]{\mathrm{ht}(#1)}
\newcommand{\rk}[1]{\mathrm{rk}(#1)}
\newcommand{\wid}{\mathrm{wd}}

\begin{document}

\author[Agrawal et al.]{Shashwat Agrawal, Amit Kuber \and Esha Gupta}
\address{Department of Mathematics and Statistics\\ Indian Institute of Technology, Kanpur\\Uttar Pradesh, India}

\email{shashuiitk@gmail.com, askuber@iitk.ac.in, esha219gupta@gmail.com}
\email{Corresponding author: askuber@iitk.ac.in, ORCID ID: 0000-0003-4812-1234}
\title{{Euclidean algorithm for a class of linear orders}}


\begin{abstract}
Borrowing inspiration from Marcone and Mont\'{a}lban's one-one correspondence between the class of signed trees and the equimorphism classes of indecomposable scattered linear orders, we find a subclass of signed trees which has an analogous correspondence with equimorphism classes of indecomposable finite rank discrete linear orders.

We also introduce the class of \emph{finitely presented linear orders}-- the smallest subclass of finite rank linear orders containing $\mathbf 1$, $\omega$ and $\omega^*$ and closed under finite sums and lexicographic products. For this class we develop a generalization of the Euclidean algorithm where the \emph{width} of a linear order plays the role of the Euclidean norm. Using this as a tool we classify the isomorphism classes of finitely presented linear orders in terms of an equivalence relation on their presentations using \emph{3-signed trees}.
\end{abstract}

\maketitle

\section{Introduction}
The classification of scattered linear orders up to isomorphisms is a very hard problem; however their classification up to equimorphisms is well-studied. Mont\'{a}lban \cite{montalban} introduced the notion of signed trees to study equimorphism classes of scattered linear orders. Together with Marcone he proved \cite[Lemma~2.8]{marcone} that the class of signed trees, $\ST$, is in one-one correspondence with the class of equimorphism classes of indecomposable linear orders, $\ILO$; this correspondence restricts to one between the class of finite signed trees, $\ST_\omega$, and the class of equimorphism classes of finite rank indecomposable linear orders, $\ILO_\omega$. 

Our interest to investigate finite rank linear orders, specially the finite rank discrete linear orders, stems from the study of chains in certain posets, known as hammocks, which were introduced by Brenner \cite{BRN} in the study of the representation theory of finite dimensional algebras. The simplest version of a hammock in the context of string algebras is a bounded discrete linear order \cite[\S~2.5]{SchTh}. Representation-theoretic literature talks about the dimension of a modular lattice \cite{prestbook}, which when restricted to linear orders is exactly its Hausdorff rank \cite{hausdorff}; nevertheless there does not seem to be any mention of the Hausdorff rank in the representation-theoretic literature.

With the goal of understanding discrete linear orders we obtain a one-one correspondence (Theorems \ref{t1} and \ref{dILOAST}) between a subclass $\AST$ of finite signed trees, consisting of \emph{alternating signed trees}, and the equimorphism classes of finite rank indecomposable discrete linear orders, $\dILO_\omega$. In the process we document and use a characterization of discrete linear orders (Proposition \ref{p2}) which we believe is known to experts, but whose proof could not be found in the literature.

Since equimorphism is a very coarse relation and can relate profoundly different linear orders, we focus on the isomorphism relation in the latter half of the paper by restricting our attention to a much smaller class consisting of \emph{finitely presented} linear orders, $\LOfp$--such orders can be written using finitely many sum and lexicographical product operations. Borrowing inspiration from signed trees, we introduce the class $3\ST_\omega$ of \emph{3-signed trees} so that there is a many-one correspondence between $3\ST_\omega$ and the isomorphism classes in $\LOfp$. Again this correspondence restricts (Theorem \ref{ATSTconverse}) to one between the class $\dLOfpb$ of bounded discrete finitely presented linear orders and a subclass $\ATST$ of $3\ST_\omega$ consisting of \emph{alternating 3-signed trees}. Sardar and the second author proved \cite{SK} that the class $\dLOfpb$ is precisely the class of hammock linear orders for domestic string algebras.

Later we introduce an equivalence relation on $3\ST_\omega$, which we call `L-equivalence', so that two L-equivalent 3-signed trees correspond to isomorphic linear orders. The main goal of this paper is to prove Corollary \ref{main} which states that two 3-signed trees are L-equivalent if and only if their corresponding linear orders are isomorphic. To this end we define the \emph{width} of a finitely presented linear order--this isomorphism-invariant plays the role of the Euclidean norm in a generalisation of Euclid's division lemma (Lemma \ref{euclid}). At the heart of the proof of the main result lies this generalisation of the Euclidean algorithm, where we use the above lemma successively to reduce the problem to the lower (Hausdorff) rank cases.

The paper is organized as follows. In \S\ref{flo} we recall some preliminaries needed for our purpose while \S\ref{slo} deals with scattered linear orders and signed trees. In \S\ref{sectAST}, we show the correspondence between $\dILO_\omega$ and $\AST$. After introducing the classes $\LOfp$ and $3\ST_\omega$ in \S\ref{fplo} we define the notion of L-equivalence in \S\ref{leq3st}. We study isomorphism classes of bounded discrete finitely presented linear orders in \S\ref{secBDFPLO}. The Euclid's division lemma is the highlight of \S\ref{edfplo} and we prove the main result in \S\ref{eallofp}.

\section*{Acknowledgements} The second author thanks Shantanu Sardar for preliminary discussions on signed trees. All authors thank Nupur Jain for careful reading of the first draft of the paper. The authors did not receive support from any organization for the submitted work.

\subsection*{Data Availability Statement}
Data sharing not applicable to this article as no datasets were generated or analysed during the current study.

\section{Fundamentals of linear orders}\label{flo}
In this section we recall standard facts about linear orders and set up notations.

Let $\mathbf n$ denote the linear order with $n$ elements for a non-negative integer $n$. Let $\Z$ and $\Q$ denote the sets of integers and rationals with usual orders. Let $\omega$ denote the first infinite ordinal. For any linear order $L$, the notation $L^*$ denotes the same underlying set with the reverse order. Set $\omega^{+1}=\omega^+:=\omega,\omega^{-1}=\omega^-:=\omega^*,\omega^0:=\mathbf1$.

There are two natural associative and non-commutative binary operations on linear orders, namely sum ($+$) and lexicographic product ($\times$). The sum operation can be extended to a family of linear orders indexed by a linear order.

\begin{rmk}\label{rightdistr}
For linear orders $L_1,L_2,L_3$ we have $$(L_1+L_2)\times L_3\cong(L_1\times L_3)+(L_2\times L_3);$$ the other distributive law fails.
\end{rmk}

Say that a linear order $L$ \emph{embeds} into another linear order $L'$ if $L$ is isomorphic to a subset of $L'$ with the induced order. Say that $L\preceq L'$ if there is an embedding of $L$ into $L'$. Say that $L$ and $L'$ are \emph{equimorphic}, denoted $L\sim L'$, if $L\preceq L'$ and $L'\preceq L$.

For linear orders $L,L'$, the former is said to be a \emph{prefix} of the latter if $L'=L+L_0$ for some linear order $L_0$. Similarly $L$ is said to be a \emph{suffix} of $L'$ if $L'=L_0+L$ for some linear order $L_0$. We say that a subset $S \subseteq L$ for a linear order is \emph{convex} if for all $a,b,c\in L$ whenever $a,b \in S$ and $a<c<b$ then we have $c\in S$.

Recall that a linear order $L$ is said to be \emph{indecomposable} if whenever $L\preceq A+B$, for some linear orders $A,B$, either $L\preceq A$ or $L\preceq B$.

\begin{rmk}\label{indecinvequim}
Indecomposability of linear orders is invariant under equimorphism.
\end{rmk}

We are interested in a subclass of indecomposable linear orders.
\begin{definition}
Say that a linear order $L$ is \emph{irreducible} if either $L=\mathbf 1$ or $L\cong\omega^\pm\times L'$ for some linear order $L'$.
\end{definition}

\begin{proposition}\label{indec}
An irreducible linear order is indecomposable.
\end{proposition}
\begin{proof}
Since $\mathbf{1}$ is clearly indecomposable, without loss we will prove that $\omega\times L$ is indecomposable; the proof of the remaining case is dual.

By definition $\omega\times L = \sum_{n\in\omega}{L}$. So suppose $\sum_{n\in\omega}{L} \ \preceq L_1 + L_2$. If $\sum_{n\in\omega}{L} \ \preceq L_1$, we are done. If not then $L_1$ embeds only finitely many, say $k$ many, copies of $L$ for some $k \in \N$. Then $\sum_{n>{k+1}, n\in\omega}{L} \cong \sum_{n\in\omega}{L} \ \preceq L_2$ so again we are done.
\end{proof}

\begin{definition}
Given a linear order $L$ and an ordinal $\alpha$, define an equivalence relation $\sim_\alpha$ on
$L$ by transfinite recursion as follows.

Let $\sim_0$ be the identity relation on $L$. For an ordinal $\alpha>0$ suppose we have defined relations $\sim_\beta$ for all $\beta<\alpha$. For $x< y$ in $L$, set $x \sim_\alpha y$ if for some $\beta < \alpha$, there are only finitely many $\sim_\beta$-equivalence classes intersecting the interval $[x,y]$. The $\sim_\alpha$-equivalence classes are convex subsets of $L$--such classes, ordered in the obvious way, constitute a linear order, denoted $L^{(\alpha)}$. The \emph{Hausdorff rank} (or, just \emph{rank}) of $L$, $\rk{L}$, is the least ordinal $\alpha$ such that $L^{(\alpha)}$ is finite. If no such $\alpha$ exists, set $\rk{L} := \infty$.
\end{definition}

\begin{rmk}\label{sumderivative}
Suppose $L=L_1+L_2$. Then $L_1^{(1)}$ (resp. $L_2^{(1)}$) is a prefix (resp. suffix) of $L^{(1)}$. Moreover, $L^{(1)}=L_1^{(1)}+L_2^{(1)}$ if and only if either $L_1$ is not bounded above or $L_2$ is not bounded below. When $L_1$ is bounded above and $L_2$ is bounded below we obtain $L^{(1)}$ by identifying the maximal element of $L_1^{(1)}$ with the minimal element of $L_2^{(1)}$ in $L_1^{(1)}+L_2^{(1)}$.
\end{rmk}

Using the above remark repeatedly we get the following.
\begin{proposition}\label{irredderiv}
If $L$ is an irreducible linear order isomorphic to $\omega^\delta\times L'$ for some $\delta\in\{+,-\}$ and some $L'$ with $\rk{L'}>0$ then $L^{(1)}$ is also isomorphic to $\omega^\delta\times L''$ for some $L''$.
\end{proposition}

\begin{proof}
Without loss suppose $L\cong\omega\times L'$. Let $L_1$ be the $\sim_1$-equivalence class of the minimal element of $L'$, if exists ($L_1 = \emptyset$ if minimal element doesn't exist), and let $L'=L_1+L_2$. This shows that if $L_2$ is non-empty then either $L_1$ is unbounded above or $L_2$ is unbounded below. Thus Remark \ref{sumderivative} gives that $(L_1 + L_2)^{(1)} = L_1^{(1)} + L_2^{(1)}$. The same remark also gives that $(L_2+L_1)^{(1)}=L'_2+L_1^{(1)}$, where $L'_2$ is prefix of $L_2^{(1)}$ that contains at most one point less than the latter. Since $\omega\times L'\cong L_1+\omega\times(L_2+L_1)$, it is easy to see that 
\begin{eqnarray*}
L^{(1)}&\cong&(\omega\times L')^{(1)}\\&\cong&L_1^{(1)}+(\omega\times(L_2+L_1))^{(1)}\\&\cong&L_1^{(1)}+\omega\times(L_2+L_1)^{(1)}\\&\cong&L_1^{(1)}+\omega\times(L'_2+L_1^{(1)})\\&\cong&\omega\times(L_1^{(1)}+L'_2).
\end{eqnarray*}
\end{proof}

Proposition \ref{irredderiv} together with induction gives the following.
\begin{corollary} \label{rk-1 derivative of irr}
If $L$ is an irreducible linear order with $1\leq\rk L<\omega$ and $L$ is isomorphic to $\omega^\delta\times L'$ for some $\delta\in\{+,-\}$ and some linear order $L'$ then $L^{(\rk L-1)}\cong\omega^\delta$.
\end{corollary}

\section{Scattered Linear Orders}\label{slo}

Recall that a linear order $L$ is \emph{scattered} if $\mathbb{Q}\npreceq L$. There is a characterization of scattered linear orders in terms of their Hausdorff ranks due to Hausdorff.
\begin{lemma} \cite{hausdorff}
A linear order $L$ is scattered if and only if $\rk{L}\neq\infty$.
\end{lemma}

The classes of scattered linear orders and finite rank linear orders are closed under finite sums and finite products.

Let $\ILO$ denote the class of countable indecomposable scattered linear orders, $\dLO$ denote the class of countable discrete linear orders, and $\dILO:=\dLO\cap\ILO$. We add a subscript $\omega$ to any class of linear orders to denote its subclass of orders with finite rank.

In a seminal work Laver \cite{laver} settled in the affirmative the conjectures of Fra\"{i}ss\'{e} \cite{fraisse} stating that indecomposables are the building blocks of the class of scattered linear orders up to equimorphism.

\begin{theorem}\label{t2}\cite{laver} Every scattered linear order can be written as a finite sum of indecomposable ones. 

Every indecomposable linear order can be written either as an $\omega$-sum or as an $\omega^*$-sum of indecomposable linear orders of smaller rank.
\end{theorem}

An important tool in the study of scattered linear orders is the notion of a signed tree introduced by Montalb\'{a}n. 
\begin{definition} \cite[Definition~2.1]{montalban}
Let $\omega^{<\omega}$ denote the set of all finite sequences in $\omega$. A signed tree is a pair $(T,s_T)$ where $T \subseteq \omega^{<\omega}$ is a non-empty well-founded tree and $s_T : T \to \{+,-\}$ is a map.
\end{definition}

The class of signed trees (resp. finite signed trees) is denoted by $\ST$ (resp. $\ST_\omega$). To talk about signed trees we use relevant notations from \cite{marcone} and \cite{montalban}. The significance of signed trees is captured by the next result.
\begin{proposition}\cite[Lemma~2.8]{marcone}\label{map}
There is a map $\mathrm{lin} : \ST \to \ILO$ such that for each infinite $L \in \ILO$, there exists $(T,s_T) \in \ST$ such that $\lin{T}{s_T} \sim L$.
\end{proposition}

Combining the above with \cite[Lemma~2.6]{marcone} we get a stronger result for finite rank indecomposables.
\begin{proposition}\label{ILO}
For each infinite $L\in\ILO$ with $\rk{L}<\omega$, there exists $(T,s_T) \in \ST_\omega$ such that $\lin{T}{s_T} \sim L$. 
\end{proposition}

From Theorem \ref{t2}, we know that an indecomposable linear order $L$ is either an $\omega$-sum or an $\omega^*$-sum. In fact, the next result states that the sign of the root of any rooted tree presenting $L$ in the sense of Proposition \ref{map} is determined--this should be known to the experts but we could not find a reference.

\begin{proposition}
Suppose $L\in \ILO$ is infinite. Then for any $(T,s_T), (T',s_{T'}) \in \ST$ with $L \sim \lin{T}{s_T}\sim \lin{T'}{s_{T'}}$ we have $s_T(\emptyset)=s_{T'}(\emptyset)$.
\end{proposition}
\begin{proof}
In view of Theorem \ref{t2}, without loss, $L$ could be written as an $\omega$-sum, say $L = \sum_{n\in\omega}{L_n}$, for indecomposables $L_n$ with $\rk{L_n}< \rk{L}$ for each $n\in\omega$. Suppose $L \sim \lin{T}{s_T}$ for some $(T,s_T) \in \ST$ with $s_T(\emptyset) = -$. Then using the language of \cite{montalban}, $\lin{T}{s_T}$ is $h$-indecomposable to the right. Since $\omega$ is a well-order and $\lin{T}{s_T} \preceq L = \sum_{n\in\omega}{L_n}$ the dual of \cite[Lemma~2.11]{montalban} gives that $L\sim\lin{T}{s_T} \preceq L_n$ for some $n\in\omega$--a contradiction to $\rk{L_n}<\rk{L}$. Therefore $s_T(\emptyset)=+$.
\end{proof}

\section{Discrete linear orders with finite rank}\label{sectAST}
Henceforth we will use a standard (re)labelling of the vertices of a finite signed tree $(T,s_T)$ described as follows. For each $\sigma\in T$ let $w(T;\sigma)$ denote the number of children of the root in $T_\sigma$. We will assume that for each vertex $\sigma$ and $x\in\omega$, $\sigma * x\in T$ if and only if $x<w(T;\sigma)$. The \emph{height} of $T\in\ST_\omega$ is defined as $\hgt{T}:=\max\{|\sigma|:\sigma\in T\}$.

We now introduce a new map $\mathrm{LIN} : \ST_\omega \to \ILO$ that is motivated from the map $\mathrm{lin}: \ST \to \ILO$.

Let $(T,s_T) \in \ST_\omega$ and $\sigma\in T$. We inductively assign a linear order to the triple $(T,s_T,\sigma)$ as follows: $$\mathrm{LIN}(T,s_T,\sigma) := \omega^{s_T(\sigma)}\times\Big(\sum_{x<w(T;\sigma)}\mathrm{LIN}(T,s_T,\sigma*x)\Big),$$ where $\sum_{x<w(T;\sigma)}\mathrm{LIN}(T,s_T,\sigma*x):=\mathbf 1$ if $w(T;\sigma)=0$. Finally we set $$\LIN{T}{s_T}:=\mathrm{LIN}(T,s_T,\emptyset).$$

Proposition \ref{indec} ensures that $\LIN{T}{s_T}$ is indeed indecomposable. The map $\mathrm{LIN}$ differs from the original map $\mathrm{lin}$ only in choosing a representative from the same equimorphism class.
\begin{proposition}\label{linLIN}
For $(T,s_T) \in \ST_{\omega}$, $\lin{T}{s_T} \sim \LIN{T}{s_T}$.
\end{proposition}
\begin{proof}
We use induction on the height of the tree to prove the result.

For the base case we have $T = \{\emptyset\}$, and the conclusion is immediate.

Next suppose $T \ne \{\emptyset\}$ and the result holds for all trees of smaller height. We prove the result when $s_T(\emptyset) = +$; the other case will be analogous. Clearly $\LIN{T}{s_T}\preceq\lin{T}{s_T}$. For the other direction we have
\begin{equation*}
    \begin{split}
        \lin{T}{s_T} &=\sum_{k\in\omega}{\Big(\sum_{n\leq\min\{k,w(T;\emptyset)\}}{\lin{T_n}{s_{T_n}}}\Big)}\\
        &=\sum_{k<w(T;\emptyset)}{\Big(\sum_{n\leq k}{\lin{T_n}{s_{T_n}}}\Big)}+\omega\times\Big(\sum_{n<w(T;\emptyset)}{\lin{T_n}{s_{T_n}}}\Big)\\
        &\sim\sum_{k<w(T;\emptyset)}{\Big(\sum_{n\leq k}{\LIN{T_n}{s_{T_n}}}\Big)}+\omega\times\Big(\sum_{n<w(T;\emptyset)}{\LIN{T_n}{s_{T_n}}}\Big)\\
        &=\sum_{k<w(T;\emptyset)}{\Big(\sum_{n\leq k}{\LIN{T_n}{s_{T_n}}}\Big)}+\LIN{T}{s_T}\\
        &\preceq(w(T;\emptyset)+\omega)\times\Big(\sum_{n<w(T;\emptyset)}{\LIN{T_n}{s_{T_n}}}\Big)\\
        &\sim\LIN{T}{s_T},
    \end{split}
\end{equation*}
where the inductive hypothesis is used in the third line.
\end{proof}

We now introduce a subclass of finite signed trees that will be used in our study of finite rank discrete linear orders.

\begin{definition}
Say that $(T,s_T) \in \ST_\omega$ is an \emph{alternating signed tree} (\emph{AST}, for short) if for every $\sigma\in T$, $w(T;\sigma)$ is even and, for each $x<w(T;\sigma)$,  $s_T(\sigma*x)=(-1)^x$.

We denote the class of alternating signed trees by $\AST$.
\end{definition}

The main goal of this section is to prove that the map $\mathrm{LIN}$ relates the classes $\AST$ and $\dILO_\omega$ in a way similar to the relation induced by the map $\mathrm{lin}$ between $\ST_\omega$ and $\ILO_\omega$ as given by Proposition \ref{ILO}.

\begin{theorem} \label{t1}
Let $(T,s_T)\in\AST$. Then $\LIN{T}{s_T} = \omega^{s_T(\emptyset)}\times L$, for a bounded $L\in\dLO_\omega$. Hence $\LIN{T}{s_T}\in\dILO_\omega$.
\end{theorem}
\begin{proof}
We show this by induction on the height of the tree.

For the base case $T =\{\emptyset\}$. Then $\LIN{T}{s_T} = \omega^{s_T(\emptyset)}\times\mathbf{1}$, where $\mathbf{1}\in\dILO_\omega$ is the singleton order.

Now suppose $T \neq \{\emptyset\}$. Note that $(T_x,s_{T_x})\in\AST$ for each $x<w(T;\emptyset)$. By the inductive hypothesis, $\LIN{T_x}{s_{T_x}}=\omega^{(-1)^x}\times L_x$ for some bounded $L_x\in\dLO_\omega$. Then $$\LIN{T}{s_T} = \omega^{s_T(\emptyset)}\times(\omega\times L_0 + \omega^*\times L_1 + \ldots \omega\times L_{w(T;\emptyset)-2} + \omega^*\times L_{w(T;\emptyset)-1}).$$

Since $L_x$ is bounded and discrete, $\omega\times L_x$ and $\omega^*\times L_x$ are discrete. Moreover $\omega\times L_x+\omega^*\times L_{x+1}$ is bounded and discrete for each even $x<w(T;\emptyset)$ because the first summand is unbounded above while the latter is unbounded below. Since finite sum of bounded discrete linear orders is again so, we get that $(\omega\times L_0 + \omega^*\times L_1 + \ldots \omega\times L_{w(T;\emptyset)-2} + \omega^*\times L_{w(T;\emptyset)-1})\in\dLO_\omega$ and is bounded. 
\end{proof} 

The rest of this section is devoted to proving Theorem \ref{dILOAST} which is the converse of the above theorem.

\begin{rmk}
If $L$ is a linear order then $\omega+L\times\Z, L\times\Z+\omega^*, \omega+L\times\Z+\omega^*, L\times\Z$ are all discrete linear orders.
\end{rmk}
In fact the converse is also true.
\begin{proposition} \label{p2}
For each discrete linear order $L$ there is a linear order $L'$ such that exactly one of the following holds.
\begin{enumerate}
    \item If $L$ has a minimum element, but not a maximum element, then $L \cong \omega + L'\times\Z$.
    \item If $L$ has a maximum element, but not a minimum element, then $L \cong L'\times\Z + \omega^*$.
    \item  If $L$ has a maximum and a minimum element, then $L \cong \omega + L'\times\Z + \omega^*$.
    \item If $L$ does not have a maximum or a minimum element, then $L \cong L'\times\Z$.
\end{enumerate}
\end{proposition}
\begin{proof}
The linear order $L'$ is obtained by removing the endpoints of the order $L^{(1)}$, if such endpoints exist, for if $L$ has a minimal (resp. maximal) element then its $\sim_1$-equivalence class is a copy of $\omega$ (resp. $\omega^*$).
\end{proof}

\begin{lemma} \label{l1}
Suppose $L\times\Z = L\times(\omega^* + \omega)$ is an indecomposable linear order for some linear order $L$. Then $L$ is indecomposable as well.
\end{lemma}
\begin{proof}
Suppose $L \preceq L_1 + L_2$ for linear orders $L_1,L_2$. Then $L\times\Z \preceq (L_1 + L_2)\times\Z \cong L_1\times\Z + L_2\times\Z$ by using right distributivity. Since $L\times\Z$ is indecomposable, without loss we assume $L\times\Z \preceq L_1\times\Z$.

Let $(a,n) \mapsto (f_1(a,n),f_2(a,n)):L\times\Z \to L_1\times\Z$ be an embedding.

\noindent{\textbf{Claim:}} The map $a \mapsto f_1(a,0):L \to L_1$ is an embedding.

Indeed if $a<b$ in $ L$ then $\omega+\omega^*$ embeds in the interval $[(a,0),(b,0)]$ in $L\times\Z$, and hence in the interval $[(f_1(a,0),f_2(a,0)),(f_1(b,0),f_2(b,0))]$ in $L_1\times\Z$. As a consequence we get $f_1(a,0)<f_1(b,0)$. Thus $L\preceq L_1$, and hence $L$ is indecomposable.
\end{proof}

Now we are ready to prove the promised analogue of Proposition \ref{ILO} for the class $\dILO_\omega$.
\begin{theorem}\label{dILOAST}
If $L\in\dILO_\omega$ is infinite, then $L \sim \LIN{T}{s_T}$ for some $(T,s_T) \in \AST$.
\end{theorem}
\begin{proof}
If $\rk{L}= 1$ then either $L\cong\omega$ or $L\cong\omega^*$--the required ASTs in those two cases are $\big(T = \{\emptyset\},s_T(\emptyset)=+\big)$ and $\big(T = \{\emptyset\},s_T(\emptyset)=-\big)$ respectively.

Now suppose $\rk{L}> 1$. By Proposition \ref{p2}, $L$ is isomorphic to one of $\omega + L'\times\Z + \omega^*$, $\omega + L'\times\Z$, $L'\times\Z + \omega^*$ or $L'\times\Z$ for some $L'$ of finite rank. Since $L$ is indecomposable, $$L\sim L'\times\Z.$$ Using Remark \ref{indecinvequim} and Lemma \ref{l1} we see that since $L$ is indecomposable, so is $L'$.

We use induction on $\rk{L}$ to prove the result. 

\noindent{Base case:} If $\rk{L} = 2$ then $\rk{L'} = 1$ and $L'$ is indecomposable. So $L'\cong \omega$ or $L'\cong \omega^*$. It can be readily verified that $\omega\times\Z$ and $\omega^*\times\Z$ are equimorphic to the images under $\mathrm{LIN}$ of the ASTs $\big(T = \{\emptyset,0,1\}, s_T = \{\emptyset\mapsto +, 0\mapsto +, 1\mapsto -\}\big)$ and $\big(T = \{\emptyset,0,1\}, s_T = \{\emptyset\mapsto -, 0\mapsto +, 1\mapsto -\}\big)$ respectively.

\noindent{Inductive case:} Suppose $m:=\rk{L}> 2$. Propositions \ref{ILO} and \ref{linLIN} together provide $(T',s_{T'})\in \ST_\omega$ such that $L'\sim\LIN{T'}{s_{T'}}$. Let $n:=w(T;\emptyset)$ and $L_x:=\LIN{T'_x}{s_{T'_x}}$ for each $x<n$.  Without loss of generality, assume that $s_{T'}(\emptyset) = +$. Then $L'\ \sim \omega\times(L_0 + \cdots + L_{n-1})$. Therefore
$$L'\times\Z\sim \ (\omega\times(L_0 + \cdots + L_{n-1}))\times\Z \ \cong \ \omega\times((L_0\times\Z) + \cdots + (L_{n-1}\times\Z)).$$
As $L' \sim \omega\times(L_0 + \cdots + L_{n-1})$, for each $x<n$ we have $1\leq\rk{L_x} < \rk{L'}=m-1$, and hence $2\leq \rk{L_x\times\Z} < m$. By the inductive hypothesis there are $(\widetilde{T}_x,s_{\widetilde{T}_x}) \in \AST$ such that $\LIN{\widetilde{T}_x}{s_{\widetilde{T}_x}} \sim L_x\times\Z$. 

Let $T \in \ST_\omega$ be defined as follows.
\begin{itemize}
    \item The root $\emptyset \in T$ has $n$ children and $s_T(\emptyset) = +$.
    \item For each child $x$ of the root, let $T_x := \widetilde{T}_x$ and the restriction of $s_T$ to $T_x$ coincides with $s_{\widetilde{T}_x}$.
\end{itemize}
Then clearly $$\LIN{T}{s_T} \sim \omega\times((L_0\times\Z) + \ldots + (L_{n-1}\times\Z)) \sim L'\times\Z.$$ 

Since $T_{x_0}\in\AST$ it only remains to edit $T$ to ensure that the signs of the children of the root of $T$ are in an alternating order starting with $+$.
\begin{enumerate}
    \item Let $(\widetilde{T},s_{\widetilde{T}})$ be a copy of $(T,s_T)$. Traverse through the children of the root of $\widetilde{T}$ in order starting from the vertex $0$.
    \item If $s_{\widetilde{T}}(x)=(-1)^x$ then move to the next child. Otherwise relabel $y$ as $y+1$ for $y\geq x$, and add a new child of the root with label $x$ and sign $(-1)^x$. As a result the width of the tree $\widetilde{T}$ increases by $1$.
    \item Continue the previous step until all the children of the root have been considered.
\end{enumerate} 

It is easy to see that the resultant $\widetilde{T}$ is in $\AST$.

\noindent{\textbf{Claim:}} $\LIN{\widetilde{T}}{s_{\widetilde{T}}}\sim L$.

To establish the claim it is enough to show that $(\widetilde{T},s_{\widetilde{T}}) \sim (T,s_T)$ in view of \cite[Lemma~2.8]{marcone}. Clearly $(T,s_T) \preceq (\widetilde{T},s_{\widetilde{T}})$. 

For the other direction, observe that the only difference between $T$ and $\widetilde{T}$ is that the latter possibly contains more children of the root than the former. So identifying the copy of the former in the latter, it remains to map the newly added vertices in a sign-preserving manner.

Since $m>2$ there is some $x_0<n$ such that $\hgt{T_{x_0}}\geq2$. Moreover since $T_{x_0}\in\AST$ there is at least one non-root vertex in $T_{x_0}$ of each sign. This provides us with the necessary vertices.

This establishes the required equimorphism and hence the claim.
\end{proof}

\section{Finitely presented linear orders}\label{fplo}
Having characterized all finite rank (discrete) linear orders up to equimorphism with the help of (A)STs, we would like to characterize these linear orders up to isomorphism. 

Recall that equimorphism is a weak notion because an equimorphism class of linear orders can consist of profoundly different linear orders (e.g., $\omega\times(\omega + \omega^*) \sim \omega\times(\omega^* + \omega)$).

Consider the following example:
$$L = \sum_{i\in\omega}{L_i} \text{ where } L_i = \begin{cases}
\omega \text{ if $i$ is prime};\\
\omega^* \text{ otherwise}.
\end{cases}$$

Clearly $L$ is a linear order of Hausdorff rank $2$ but there is no `compact' way to present it. However it can be easily seen to be equimorphic to $\omega\times(\omega + \omega^*)$, which is `compactly presented'.

To tackle the problem of characterizing linear orders up to isomorphisms we restrict our attention to a subclass of `compactly presented' linear orders.
\begin{definition}\label{fp}
The class $\LOfp$ of \emph{finitely presented} linear orders is defined as the smallest subclass of linear orders closed under isomorphisms such that
\begin{enumerate}
    \item $\mathbf{0},\mathbf{1}\in\LOfp$;
    \item if $L_1,L_2\in\LOfp$ then $L_1+L_2\in\LOfp$;
    \item if $L\in\LOfp$ then $\omega\times L,\omega^*\times L\in\LOfp$.
\end{enumerate}
\end{definition}
Clearly a finitely presented linear order is of finite rank, and all such orders are precisely those which can be written using finitely many $+$ and $\omega^{\pm}\times(\mbox{-})$ operations.

\begin{rmk}\label{fpchar}
In view of Remark \ref{rightdistr}, $\LOfp$ is precisely the smallest class of linear orders that contains $\mathbf{0},\mathbf{1},\omega,\omega^*$ and is closed under finite sums and finite products.
\end{rmk}
\begin{rmk}
If $(T,s_T)\in \ST_\omega$ then $\LIN{T}{s_T}\in\LOfp$. Hence Proposition \ref{ILO} gives that each indecomposable finite rank linear order is equimorphic to a finitely presented one.
\end{rmk}

Since $\LIN{T}{s_T}$ is indecomposable for $(T,s_T)\in\ST_\omega$ we define a new class of rooted trees to deal with decomposable linear orders as well. This would be very similar to finite signed trees except that we allow vertices to have a third sign.
\begin{definition}
A \emph{3-signed tree} (3ST, for short) is a pair $(T, s_T)$ where $T \subseteq \omega^{<\omega}$ is a non-empty well-founded finite tree and $s_T : T \to \{+,-,0\}$ is a map satisfying $s_T(\sigma)=0$ if and only if $\sigma$ is either the root or a leaf of $T$. The class of 3-signed trees will be denoted $3\ST_\omega$.
\end{definition}
For a tree $T$ and a non-root vertex $\sigma$ of $T$, we denote its parent by $\pi(\sigma)$.

There is a natural embedding $I:\ST_\omega\to3\ST_\omega$ that appends a sign $0$ root as well as a sign $0$ child to each leaf of a signed tree in $\ST_\omega$.

In a similar spirit, for $(T, s_T) \in 3\ST_\omega$ and a non-root vertex $\sigma\in T$, define $(T_\sigma,s_{T_\sigma})$ to be the $3ST$ obtained by appending a sign $0$ root to the subtree of $(T,s_T)$ induced by $\sigma$. Moreover if $s_T(\sigma)\neq0$, we also define $(\widehat T_\sigma,s_{\widehat T_\sigma})$ to be the $3ST$ obtained by changing the sign of the root of the subtree of $(T,s_T)$ induced by $\sigma$ to $0$.

Given $(T,s_T)\in3\ST_\omega$ and $\delta\in\{+,-\}$, we define $(T^\delta,s_{T^\delta})$ to be the 3ST obtained by assigning sign $\delta$ to the root of $T$, and then appending a sign 0 root to such tree.

We continue to use the standard (re)labelling of the vertices of a $3ST$ as discussed at the beginning of \S\ref{sectAST}.

For $(T,s_T),(T',s_{T'})\in 3\ST_\omega$ define $(T,s_T)\curlyvee (T',s_{T'})$, the \emph{join} of $(T,s_T)$ and $(T',s_{T'})$, to be the $3ST$ obtained by identifying the roots of $T$ and $T'$ in $(T,s_T)\sqcup (T',s_{T'})$, where the induced subtrees of the children of the root of the latter are added after the induced subtrees of the children of the root of the former. We then call $(T,s_T)$ a \emph{prefix} of $(T,s_T)\curlyvee (T',s_{T'})$ and $(T',s_{T'})$ a \emph{suffix} of $(T,s_T)\curlyvee (T',s_{T'})$.

We now associate a linear order to each element of $3\ST_\omega$ using a construction similar to the map $\mathrm{LIN}$ described for $\ST_\omega$.

Let $(T,s_T) \in 3\ST_\omega$. If $|T|=1$ then we define $$\LIN{T}{s_T}:=\mathbf 0.$$ If $|T|>1$ and $\sigma\in T$ then we inductively assign a linear order to the triple $(T,s_T,\sigma)$ as follows: $$\mathrm{LIN}(T,s_T,\sigma) := \omega^{s_T(\sigma)}\times\Big(\sum_{x<w(T;\sigma)}\mathrm{LIN}(T,s_T,\sigma*x)\Big),$$ where $\sum_{x<w(T;\sigma)}\mathrm{LIN}(T,s_T,\sigma*x):=\mathbf 1$ if $w(T;\sigma)=0$. Finally we set $$\LIN{T}{s_T}:=\mathrm{LIN}(T,s_T,\emptyset).$$

\begin{rmk}
For $(T,s_T) \in \ST_\omega$ it is readily seen that $\LIN{T}{s_T}\cong\mathrm{LIN}(I(T,s_T))$.
\end{rmk}

Clearly for $(T,s_T) \in 3\ST_\omega$, $\LIN{T}{s_T}$ will consist of finitely many $+$ and $\omega^{\pm}\times(\mbox{-})$ operations, hence $\mathrm{LIN}$ is a map $3\ST_\omega\to\LOfp$. Moreover for a finitely presented linear order $L$, the inductive nature of Definition \ref{fp} allows us to construct a 3-signed tree $(T,s_T)$ such that $\LIN{T}{s_T}\cong L$. We collect these observations in the next result.

\begin{proposition}\label{3STLOfpcorr}
For a linear order $L$, $L\in\LOfp$ if and only if there is $(T,s_T) \in 3\ST_\omega$ such that $L\cong\LIN{T}{s_T}$.
\end{proposition}

\begin{proposition}\label{summandfp}
Let $L\in\LOfp$. If $L=L_1+L_2$ for non-empty linear orders $L_1,L_2$ then $L_1,L_2\in\LOfp$.
\end{proposition}
\begin{proof}
Let $L\in\LOfp$ be non-empty. We prove the result using induction on $\rk{L}$. Proposition \ref{3STLOfpcorr} yields a 3ST $(T,s_T)$ such that $L\cong\LIN{T}{s_T}$ and $\rk{L}=\hgt{T}-1$. Let $w:=w(T;\emptyset)$.

If $\rk{L}=0$ then the conclusion is obvious. On the other hand if $\rk{L}>0$ then there are two cases.
\begin{itemize}
    \item[$(w=1)$] Here $L\cong\omega^{s_T(0)}\times\LIN{\widehat{T}_0}{s_{\widehat{T}_0}}$. If $\widetilde{L}:=\LIN{\widehat{T}_0}{s_{\widehat{T}_0}}$ then $\widetilde{L}\in\LOfp$. Without loss assume that $s_T(0)=+$. Then there exists some $n\in\omega$ and linear order $L_1'$ such that $\mathbf{n}\times\widetilde{L}\cong L_1+L_1'$ and $L_2\cong L_1'+\omega\times\widetilde{L}$. Since $\rk{\mathbf{n}\times\widetilde{L}}=\rk{\widetilde{L}}<\rk L$, we have $L_1,L_1',L_2\in\LOfp$ by the induction hypothesis.
    \item[$(w>1)$] Let $L'_i:=\LIN{T_i}{s_{T_i}}$ for $0\leq i<w$. Then $L'_i\in\LOfp$ and $L\cong L'_0+L'_1+\cdots+L'_{w-1}$. We can find an integer $0\leq j<w$ and linear orders $\widetilde{L}'_j$ and $\widetilde{L}''_j$ such that $L_1=L'_0+\cdots+L'_{j-1}+\widetilde{L}'_j$, $L_2=\widetilde{L}''_j+L'_{j+1}+\cdots+L'_{w-1}$, and $L'_j=\widetilde{L}'_j+\widetilde{L}''_j$. Since $w(T_j;\emptyset)=1$ from the above case we conclude that $\widetilde{L}'_j,\widetilde{L}''_j\in\LOfp$. Thus $L_1,L_2\in\LOfp$.
\end{itemize}
\end{proof}

\begin{proposition}\label{derfp}
If $L\in\LOfp$ then $L^{(1)}\in\LOfp$.
\end{proposition}
\begin{proof}
We will prove the result using induction on $\rk{L}$.

If $\rk{L}\leq1$ then $L^{(1)}$ is finite and hence finitely presented. On the other hand if $\rk{L}>1$ then there are two cases.

\noindent{\textbf{Case I:}} $L$ is irreducible.

Here $L\cong\omega^\delta\times\widetilde{L}$ for some $\delta\in\{+,-\}$. If $\delta=+$ then Proposition \ref{irredderiv} gives that $L^{(1)}\cong\omega\times \widetilde{L}^{(1r)},$ where $\widetilde{L}^{1r}=\widetilde{L}^{(1)}$ if $\widetilde{L}$ is not bounded and $\widetilde{L}^{(1r)}+1=\widetilde{L}^{(1)}$ if $\widetilde{L}$ is bounded.

Since $\rk{\widetilde{L}}<\rk{L}$, $\widetilde{L}^{(1)}\in\LOfp$ by induction hypothesis, and then $\widetilde{L}^{(1r)}\in\LOfp$ by Proposition \ref{summandfp}. Hence $L^{(1)}\in\LOfp$ by the definition of the class $\LOfp$.

A similar argument holds if $\delta=-$.

\noindent{\textbf{Case II:}} $L$ is not irreducible.

Here $L=L_1+L_2+\cdots+L_n$, where each $L_i\in\LOfp$ is irreducible. Using Remark \ref{sumderivative} we have $$L^{(1)}=L_1^{(1r)}+L_2^{(2r)}+\cdots+L_n^{(1)},$$ where for $1\leq i\leq n-1$ we set $$L_i^{(1r)}:=\begin{cases}L_i^{(1)}-\{\max{L_i^{(1)}}\}&\mbox{if }\max{L_i}\mbox{ and }\min{L_{i+1}}\mbox{ exist};\\L_i^{(1)}&\mbox{otherwise}.\end{cases}$$

By Case I and Proposition \ref{summandfp} each $L_i^{(1)},L_i^{(1r)}\in\LOfp$. Hence $L^{(1)}\in\LOfp$ by the definition of the class $\LOfp$.
\end{proof}

We end this section by noting an interesting observation.
\begin{proposition}\label{****}
Suppose $L,L'\in\LOfp$ and $\rk L=\rk{L'}$. If $\omega\times L$ is a prefix of $\omega\times L'$ then $\omega\times L\cong\omega\times L'$. Dually if $\omega^*\times L$ is a suffix of $\omega^*\times L'$ then $\omega^*\times L\cong\omega^*\times L'$.
\end{proposition}

\begin{proof}
Suppose $\omega\times L$ is a prefix of $\omega\times L'$. If $\omega\times L$ is a proper prefix of $\omega\times L'$ then in fact $\omega\times L$ is a proper prefix of $\mathbf{p}\times L'$ for some $p\geq 1$. This is clearly a contradiction since $\rk{\omega\times L}=\rk{L}+1>\rk{L'}=\rk{\mathbf{p}\times L'}$. Hence the proof.
\end{proof}

\section{L-equivalence on $3\ST_\omega$}\label{leq3st}
In this section we introduce two constructions on $3$-signed trees in a way that does not change $\mathrm{LIN}$.

For all linear orders $L_1,L_2,\hdots,L_n,$ the following identities hold.
\begin{equation*}
    \begin{split}
       \omega\times(L_1+L_2+\hdots+L_n)&\cong L_1+\omega\times(L_2+\hdots+L_n+L_1),\\
       \omega^*\times(L_1+L_2+\hdots+L_n)&\cong\omega^*\times(L_n+L_1+\hdots+L_{n-1})+L_n,
    \end{split}
\end{equation*}
This motivates the following construction.

\begin{construction}
Suppose $(T,s_T)\in3\ST_\omega$ and $\sigma\in T$ is a non-root non-leaf vertex. Thus $s_T(\sigma)\neq0$. Suppose $m:=w(T;\pi(\sigma))$, $n:=w(T;\sigma)$, $p:=(\frac{s_T(\sigma)-1}{2})\ (\mathrm{mod}\ n)$ and $\sigma=\pi(\sigma)*k$ for some $k<m$. Define $g(i):=(i+s_T(\sigma))\ (\mathrm{mod}\ n)$ for $0\leq i<n$.

We construct a new $3$-signed tree $\mathsf{EXUDE}((T,s_T);\sigma)$, that has either the first or the last child subtree of $\sigma$ exuded out depending on its sign, as below.
\begin{enumerate}
    \item Suppose $(T',s_{T'})$ is a copy of $(T,s_T)$.
    \item Reassign $(T'_{\sigma*i},s_{T'_{\sigma*i}})$ to be a copy of $(T_{\sigma*g(i)},s_{T_{\sigma*g(i)}})$ for $i<n$.
    \item Set $(T'_{\pi(\sigma)*m},s_{T'_{\pi(\sigma)*m}})$ to be a copy of $(T_{\sigma*p},s_{T_{\sigma*p}})$.
    \item Permute in a cyclic order with step-size $1$ the induced subtrees $(T'_{\pi(\sigma)*i},s_{T'_{\pi(\sigma)*i}})$ with indices between $k\leq i\leq m$ when $s_T(\sigma)=+$, and $k<i\leq m$ when $s_T(\sigma)=-$.
    \item Set $\mathsf{EXUDE}((T,s_T);\sigma):=(T',s_{T'})$.
\end{enumerate}
\end{construction}
It is readily verified that $\LIN{T}{s_T}\cong\mathrm{LIN}(\mathsf{EXUDE}((T,s_T);\sigma))$.

For a linear order $L$, $m\geq 1$ and $\delta\in\{+,-\}$, we also have $$\omega^\delta\times L\cong\omega^\delta\times(\mathbf{m}\times L).$$
This motivates the following construction.
\begin{construction}
Suppose $(T,s_T)\in3\ST_\omega$, $\sigma\in T$ is a non-root non-leaf vertex, and $m\geq 1$. Thus $s_T(\sigma)\neq0$. Let $n:=w(T;\sigma)$.

We construct a new $3$-signed tree $m\mbox{-}\mathsf{REPL}((T,s_T);\sigma)$, that has $m$ copies of the children of $\sigma$, as below.

\begin{enumerate}
    \item Suppose $(T',s_{T'})$ is a copy of $(T,s_T)$.
    \item For $i<n$ and $1\leq j<m$, set $(T'_{\sigma*(jn+i)},s_{T'_{\sigma*(jn+i)}})$ to be a copy of $(T_{\sigma*i},s_{T_{\sigma*i}})$.
    \item Set $m\mbox{-}\mathsf{REPL}((T,s_T);\sigma):=(T',s_{T'})$.
\end{enumerate}
\end{construction}
Again it is readily verified that $\LIN{T}{s_T}\cong\mathrm{LIN}(m\mbox{-}\mathsf{REPL}((T,s_T);\sigma))$.

\begin{definition}
Define a relation $\approx_L$ on $3\ST_\omega$ as follows.

For $(T,s_T)\in3\ST_\omega$, $\sigma\in T$ with $s_T(\sigma)\neq0$, and $m\geq 1$,
\begin{itemize}
    \item $(T,s_T)\approx_L\mathsf{EXUDE}((T,s_T);\sigma)$;
    \item $(T,s_T)\approx_L m\mbox{-}\mathsf{REPL}((T,s_T);\sigma)$.
\end{itemize}
Say that \emph{L-equivalence} on $3\ST_\omega$ is the equivalence relation generated by $\approx_L$, which we again denote by $\approx_L$.
\end{definition}

\begin{rmk}\label{Leqcong}
L-equivalence is a congruence relation on $3\ST_\omega$. Suppose $(T,s_T)$, $(T',s_{T'})$ and $(T'',s_{T''})$ are $3\ST$s and $\delta\in\{1,-1\}$. If $(T',s_{T'})\approx_L(T'',s_{T''})$ then
\begin{itemize}
    \item $T\curlyvee T'\approx_L T\curlyvee T''$ and $T'\curlyvee T\approx_L T''\curlyvee T$;
    \item ${T'}^\delta\approx_L{T''}^\delta$.
\end{itemize}
\end{rmk}

We obviously have the following.
\begin{proposition}\label{LequivgivesLINiso}
Suppose $(T,s_T),(T',s_{T'})\in3\ST_\omega$ and $(T,s_T)\approx_L(T',s_{T'})$. Then $\LIN{T}{s_T}\cong\LIN{T'}{s_{T'}}$.
\end{proposition}

\section{Bounded discrete finitely presented linear orders}\label{secBDFPLO}
Let $\dLOfpb$ denote the subclass of $\LOfp$ consisting of bounded discrete finitely presented linear orders. Recall that we defined in \S\ref{sectAST} a subclass of finite signed trees, namely the alternating signed trees ($\AST$) which corresponded to the class of indecomposable discrete linear orders up to equimorphism (Theorem \ref{dILOAST}). In this section we define alternating 3-signed trees, which will be a subclass of 3-signed trees and show in Theorem \ref{ATSTconverse} that such trees correspond to the class of bounded discrete finitely presented linear orders up to isomorphism (cf. Proposition \ref{3STLOfpcorr}).
\begin{definition}
Say that $(T,s_T) \in 3\ST_\omega$ is an \emph{alternating 3-signed tree} if for each non-leaf vertex $\sigma\in T$ exactly one of the following holds:
\begin{itemize}
    \item $s_T(\sigma*x)=0$ for each $x<w(T;\sigma)$;
    \item $w(T;\sigma)$ is even and $s_T(\sigma*x)=(-1)^x$ for each $x<w(T;\sigma)$.
\end{itemize}
\end{definition}

We denote the subclass of $3\ST_\omega$ consisting of alternating 3-signed trees by $\ATST$.

\begin{theorem}\label{ATSTconverse}
Suppose $L\in\LOfp$. Then $L\in\dLOfpb$ if and only if there is $(T,s_T)\in\ATST$ such that $\LIN{T}{s_T}\cong L$.
\end{theorem}
\begin{proof}
If $(T,s_T)\in\ATST$ then clearly $\LIN{T}{s_T}\in\dLOfpb$.

For the other direction suppose $L\in\dLOfpb$. Then Proposition $\ref{p2}$ gives that $L\cong\omega + L'\times\Z + \omega^*$ for some linear order $L'$. Since $L$ is finitely presented so is $L'$ by Propositions \ref{summandfp} and \ref{derfp}. Then Proposition \ref{3STLOfpcorr} yields $(T',s_{T'}) \in 3\ST_\omega$ such that $L'\cong\LIN{T'}{s_{T'}}$. We construct another tree $(T,s_T)$ such that $\LIN{T}{s_T}\cong L'\times\Z$ as follows.

Starting with $(T,s_T)$ as a copy of $(T',s_{T'})$, duplicate all leaf vertices while ensuring that such duplicate copies are adjacent to each other and the relative position of each of the duplicate copies with respect to its siblings is unaltered. If $\sigma*x$ and $\sigma*(x+1)$ are duplicates then assign $s_T(x)=-$ and $s_T(x+1)=+$. Furthermore, to maintain our convention of sign $0$ leaves, we add one child for both $\sigma* x$ and $\sigma*(x+1)$ with sign $0$. It is readily seen that $\LIN{T}{s_T}\cong L'\times\Z$ as required.

Now we find $(\widetilde{T},s_{\widetilde{T}})\approx_L(T,s_T)$ that is very close to being in $\ATST$. In view of Proposition \ref{LequivgivesLINiso} we will have $\LIN{\widetilde{T}}{s_{\widetilde{T}}}\cong L'\times\Z$.

Let $h:=\hgt{T}$. If $h=2$ then $\hgt{T'}=1$, i.e., $L'=\mathbf{n}$ for some $n\in\omega$ so that $L\cong\mathbf{n}\times\Z$. In this case let $(\widetilde{T},s_{\widetilde{T}}):=(T,s_T)$.

On the other hand if $h>2$ then, for $2\leq i\leq h-1$, let $S_i:=\{\sigma\in T\mid\hgt{\widehat{T}_\sigma}=i\}$ and $t_i:=|S_i|$. Order each $S_i$ using lexicographic ordering on its vertices thought of as finite subsets of $\omega$. Let $S:=\{(i,j)\mid2\leq i\leq h-1,\ 1\leq j\leq t_i\}\cup\{(1,1)\}$ be equipped with lexicographic order $\lex$ on the pairs. Set $t_1:=1$. For $(i,j)(\neq(1,1))\in S$ we denote the corresponding vertex of $T$ by $\sigma_{i,j}$. For each $(i,j)\in S$ we construct $T^{i,j}\in 3\ST_\omega$ and for each $(i,j)\lex(i',j')$ in $S$ we describe a height and sign preserving embedding $f^{i',j'}_{i,j}:T^{i,j}\to T^{i',j'}$. For brevity we denote by $f^{i,j}$ the map $f^{i,j}_{1,1}$ for each $(i,j)\in S$.
\begin{enumerate}
    \item Set $T^{1,1}:=T$ and $f^{1,1}$ to be the identity map.
    \item Suppose $T^{i,j}$ is constructed and the immediate successor $(i',j')$ of $(i,j)$ in $S$ exists. Set $T^{i',j'}:=\mathsf{EXUDE}((T^{i,j},s_{T^{i,j}});f^{i,j}(\sigma_{i',j'}))$ and $f^{i',j'}_{i,j}:T^{i,j}\to T^{i',j'}$ to be the canonical inclusion. For each $(i'',j'')\lex(i,j)$ in $S$ set $f^{i',j'}_{i'',j''}:=f^{i',j'}_{i,j}\circ f^{i,j}_{i'',j''}$.
    \item Finally set $(\widetilde{T},s_{\widetilde{T}}):=(T^{h-1,t_{h-1}},s_{T^{h-1,t_{h-1}}})$.
\end{enumerate}
Clearly $(\widetilde{T},s_{\widetilde{T}})\approx_L(T,s_T)$ because the former is obtained by a sequence of $\mathsf{EXUDE}$ routines.

For brevity let $T^i:=T^{i,t_i}$ for $1\leq i<h$ and $f_i^{i'}:=f_{i,t_i}^{i',t_{i'}}$ for $1\leq i<i'<h$.

We show the following for $T^i$ using induction for each $1\leq i\leq h-1$:
\begin{enumerate}
    \item[$(A_i)$] if $\sigma\in T^i$ and $1<\hgt{\widehat{T^i_\sigma}}\leq i$ then $\widehat{T^i_\sigma}\in\ATST$;
    \item[$(B_i)$] if $\sigma\in T^i$ and $\hgt{\widehat{T^i_\sigma}}=i+1$ then $w(T^i;\sigma)$ is even and $s_{T^i}(\sigma*x)=(-1)^{x+1}$ for each $x<w(T^i;\sigma)$.
\end{enumerate}

From the construction of $T$ it is clear that $(B_1)$ holds and $(A_1)$ holds vacuously.

Assume for induction that, for some $1\leq i<h-1$, the statements $(A_i)$ and $(B_i)$ hold.

To see that $(A_{i+1})$ holds suppose $\sigma\in T^{i+1}$ and $1<\hgt{\widehat{T^{i+1}_\sigma}}\leq i+1$. There are three possibilities.
\begin{itemize}
    \item If $\hgt{\widehat{T^{i+1}_\sigma}}\leq i$ and $\sigma=f_i^{i+1}(\sigma')$ for some $\sigma'\in T^i$ then $\widehat{T^{i+1}_\sigma}$ is isomorphic to $\widehat{T^i_{\sigma'}}$. Since $(A_i)$ gives that the latter is in $\ATST$ we see that the former is also in $\ATST$.
    \item If $\hgt{\widehat{T^{i+1}_\sigma}}\leq i$ and $\sigma\notin\mathrm{Im}(f_i^{i+1})$ then the construction of $\mathsf{EXUDE}$ operation gives some $\sigma'\in T^i$ such that $\widehat{T^{i+1}_\sigma}$ is isomorphic to $\widehat{T^i_{\sigma'}}$. Thus the conclusion follows as in the above item.
    \item If $\hgt{\widehat{T^{i+1}_\sigma}}=i+1$ and $\sigma=f_i^{i+1}(\sigma')$ for some $\sigma'\in T^i$ then $(B_i)$ guarantees that $w(T^i;\sigma')$ is even and that $s_{T^i}(\sigma'*x)=(-1)^{x+1}$ for each $x<w(T^i;\sigma')$. Since there is $1\leq k<t_{i+1}$ such that $T^{i+1,k+1}=\mathsf{EXUDE}((T^{i+1,k},s_{T^{i+1,k}});f^{i+1,k}_{i,t_i}(\sigma'))$, we get that $\widehat{T^{i+1}_\sigma}\in\ATST$.
\end{itemize}

Now we show that $(B_{i+1})$ holds. The map $f^{i+1}$ restricts to a bijection between $S_{i+2}$ and the set $\{\sigma\in T^{i+1}\mid\hgt{\widehat{T^{i+1}_\sigma}}=i+2\}$. Choose an element $\sigma$ from the latter set and $x<w(T^{i+1};\sigma)$. Let $j:=\hgt{\widehat{T^{i+1}_{\sigma*x}}}$ and $\delta:=s_{T^{i+1}}(\sigma*x)$. Since at no step in the construction of $\widetilde T$, the $\mathsf{EXUDE}$ operation is applied at a vertex $\sigma''\in T^{m,n}$ with $\hgt{\widehat{T^{m,n}_{\sigma''}}}=1$, we see that $1\leq j\leq i+1$.

If $\sigma*x=f^{i+1}(\sigma_{j,k})$ for some $(j,k)\in S$ then one of the following happens.
\begin{itemize}
    \item[$(\delta=+)$] If $f^{j,k}(\sigma_{j,k})\in T^{j,k}$ is of the form $\sigma'*y$ then the $\mathsf{EXUDE}$ construction and $(B_{j-1})$ together ensure that $\sigma'*(y-1)$ exists and $s_{T^{j,k}}(\sigma'*(y-1))=-$. Since for each $(j,k)\lex(i',k')\leq_{\mathrm{lex}}(i+1,t_{i+1})$, the vertices $f^{i',k'}_{j,k}(\sigma'*(y-1))$ and $f^{i',k'}_{j,k}(\sigma'*y)$ are immediate siblings and their relative position is also unaltered, we conclude that $\sigma*(x-1)\in T^{i+1}$ and $s_{T^{i+1}}(\sigma*(x-1))=-$.
    \item[$(\delta=-)$] As above we can argue that $\sigma*(x+1)\in T^{i+1}$ and $s_{T^{i+1}}(\sigma*(x+1))=+$.
\end{itemize}

If $\sigma*x\notin\mathrm{Im}(f^{i+1})$ then let $(j,k)\in S$ be the minimum such that $\sigma*x=f^{i+1,t_{i+1}}_{j,k}(\sigma'')$ for some $\sigma''$. Since $(1,1)\lex(j,k)$, the immediate predecessor $(j',k')$ of $(j,k)$ in $S$ exists. Since $T^{j,k}=\mathsf{EXUDE}((T^{j',k'},s_{T^{j',k'}});f^{j',k'}(\sigma_{j,k}))$ and $\pi(\sigma*x)\in\mathrm{Im}(f^{i+1})$, we conclude that $\sigma''$ and $f^{j,k}(\sigma^{j,k})$ are immediate siblings.

Suppose $\sigma''=\pi(\sigma'')*z$. If $\delta=+$ then it follows using $(B_{j-1})$ that $s_{T^{j,k}}(f^{j,k}(\sigma^{j,k}))=-$ and $f^{j,k}(\sigma^{j,k})=\pi(\sigma'')*(z-1)$. Similarly if $\delta=-$ then $s_{T^{j,k}}(f^{j,k}(\sigma^{j,k}))=+$ and $f^{j,k}(\sigma^{j,k})=\pi(\sigma'')*(z+1)$.

This completes the proof of $(B_{i+1})$.

Let $(T_1,s_{T_1}),(T_2,s_{T_2})\in3\ST_\omega$ be defined by $T_1=T_2:=\{\emptyset,0,00\}$ and where $s_{T_1}(0)=+$ and $s_{T_2}(0)=-$. Then define $(\overline T,s_{\overline T}):=(T_1,s_{T_1})\curlyvee(T^{h-1},s_{T^{h-1}})\curlyvee(T_2,s_{T_2})$. The statement $(B_{h-1})$ ensures $(\overline T,s_{\overline T})\in\ATST$ while Proposition \ref{LequivgivesLINiso} and Remark \ref{Leqcong} ensure that $\LIN{\overline T}{s_{\overline T}}\cong \omega+L'\times(\omega^*+\omega)+\omega^*\cong L$.
\end{proof}

\section{Euclidean division with finitely presented linear orders}\label{edfplo}
In this section we gather some tools to prove the converse of Proposition \ref{LequivgivesLINiso} regarding the notion of the width of finitely presented linear orders defined below.
\begin{definition}
For $L\in\LOfp$ define the \emph{width} of $L$ (denoted $\wid(L)$) to be the minimum value of $w(T;\emptyset)$ where $(T,s_T)\in3\ST_\omega$ and $\LIN{T}{s_T}\cong L$.
\end{definition}
Given $L\in\LOfp$, $L$ is irreducible if and only $\wid(L)=1$. Throughout the rest of this paper we say that $L\in\LOfp$ is an $\omega$-sum (resp. $\omega^*$-sum) if $L\cong\omega\times L'$ (resp. $L\cong\omega^*\times L'$) for some $L'$.

\begin{lemma}\emph{(Irreducible affix lemma)}\label{affixirred}
Let $n>1$ and $L_i\in\LOfp$ such that $\wid(L_i)=1$ for $1\leq i\leq n$, and $L_1+L_2+\hdots+L_n$ is an $\omega$-sum then for each $1\leq i<n$ $\rk{L_n}>\rk{L_i}$, $L_i+L_{i+1}+\hdots+L_n$ is an $\omega$-sum, and hence $\wid{(L_i+L_{i+1}+\hdots+L_n)}=1$.
\end{lemma}

\begin{proof}
Let $f:(L_1+L_2+\hdots+L_n)\to\omega\times L$ be an isomorphism. Since $f(L_1+L_2+\hdots+L_{n-1})$ is a proper prefix of $\omega\times L$, there exists a smallest $m\geq 1$ such that it is a prefix of $\mathbf{m}\times L$. Clearly $\omega\times L$ is isomorphic to a suffix of $L_n$. Hence $$\rk{L_1+L_2+\hdots+L_{n-1}}\leq\rk{\mathbf{m}\times L}<\rk{\omega\times L}=\rk{L_1+L_2+\hdots+L_n}=\rk{L_n}.$$

Fix $1\leq i<n$. Let $\mathbf{m}\times L=f(L_1+L_2+\hdots+L_{n-1}+\bar L)$ for some prefix $\bar L$ of $L_n$. Then
\begin{align*}
f(L_1+L_2+\hdots+L_n)&\cong\omega\times(\mathbf{m}\times L)\\&=\omega\times f(L_1+L_2+\hdots+L_{n-1}+\bar L)\\&\cong f(L_1+\hdots+L_{i-1})+\omega\times f(L_i+\hdots L_{n-1}+\bar L+L_1+\hdots+L_{i-1}).
\end{align*}

Cancelling $f(L_1+\hdots+L_{i-1})$ from both sides and applying $f^{-1}$, we obtain the result.
\end{proof}

The next result states that any suffix of an $\omega$-sum is so too.
\begin{proposition}\label{omegasumaffix}
Let $L,L'\in\LOfp$. If $L'=L+\tilde L$ for some non-empty $\tilde L$ and $L'$ is an $\omega$-sum then there exists $\bar L\in\LOfp$ such that $L'\cong\omega\times(L+\bar L)$ and $\tilde L\cong\omega\times(\bar L+L)$.
\end{proposition}

\begin{proof}
Suppose $L'\cong\omega\times\tilde L'$ for a prefix $\tilde L'$ of $L'$. Since $L$ is a proper prefix of $L'$, there exists an $m\geq1$ such that $L$ is a proper prefix of $\mathbf{m}\times\tilde L'$. Then $\mathbf{m}\times\tilde L'\cong L+\bar L$ for some finitely presented $\bar L$. Thus $$L+\tilde L=L'\cong\omega\times\tilde L'\cong\omega\times(\mathbf{m}\times\tilde L')\cong\omega\times(L+\bar L)\cong L+\omega\times(\bar L+L).$$ Since each isomorphism in the above line preserves the prefix $L$, the result follows by cancelling a copy of $L$ from both sides.
\end{proof}

Using the above result we can patch two irreducible $\omega$-sums.
\begin{corollary}\label{patchinglemma}
Suppose $L_1,L_2,L'\in\LOfp$ are irreducible such that $L_1$ is a proper prefix of $L'$ and $L'$ is a proper prefix of $L_1+L_2$. If $L_1$ is an $\omega$-sum then $L'$ is also an $\omega$-sum. Moreover $\wid(L_1+L_2)=1$.
\end{corollary}

\begin{proof}
Let $L'=L_1+L'_2$ and $L_2=L'_2+L''_2$, where $L'_2$ and $L''_2$ are non-empty. Since $L'$ is irreducible, it is either an $\omega$-sum or an $\omega^*$-sum. Since $L_1$ is a prefix of of $L'$ and $L_1$ is an $\omega$-sum, the dual of the irreducible affix lemma (Lemma \ref{affixirred}) gives that $L'$ is also an $\omega$-sum. The same lemma also gives that $L'_2$ is an $\omega$-sum, which gives that $L_2$ is an $\omega$-sum.

Applying Proposition \ref{omegasumaffix} to the prefix embedding of $L'_2$ in $L_2$ gives that $L_2\cong\omega\times(L'_2+\bar L'_2)$ for some $\bar L'_2\in\LOfp$. The same proposition applied to the prefix embedding of $L_1$ in $L'$ gives that $L'\cong\omega\times(L_1+\bar L_1)$ and $L'_2\cong\omega\times(\bar L_1+L_1)$ for some $\bar L_1\in\LOfp$. Then
\begin{align*}
L_1+L_2&\cong L_1+\omega\times(L'_2+\bar L'_2)\\&\cong L_1+\omega\times(\omega\times(\bar L_1+L_1)+\bar L'_2)\\&\cong L_1+\omega\times(\bar L_1+\omega\times(L_1+\bar L_1)+\bar L'_2)\\&\cong L_1+\bar L_1+\omega\times(L_1+\bar L_1)+\omega\times(\bar L'_2+\bar L_1+\omega\times(L_1+\bar L_1))\\&\cong\omega\times(L_1+\bar L_1)+\omega\times(\bar L'_2+\bar L_1+\omega\times(L_1+\bar L_1))\\&\cong\omega\times(\omega\times(L_1+\bar L_1)+\bar L'_2+\bar L_1),
\end{align*}
which gives that $\wid(L_1+L_2)=1$.
\end{proof}

Now we explore the width of finite sums of finitely presented linear orders using that of the sum of consecutive pairs.

\begin{rmk}\label{widthofpairs}
Let $L\in\LOfp$. If $n:=\wid(L)>1$ and $L=L_1+L_2+\hdots+L_n$, where each $L_i$ is irreducible then $\wid(L_i+L_{i+1})=2$ for $1\leq i<n$.
\end{rmk}

In fact the converse of the above remark is also true.
\begin{theorem}\label{totalwidthfrompairs}
Let $n>1$ and $L_i\in\LOfp$ for $1\leq i\leq n$. If $\wid(L_i)=1$ for $1\leq i\leq n$, and $\wid(L_i+L_{i+1})=2$ for $1\leq i<n$ then $\wid(L_1+L_2+\hdots+L_n)=n$.
\end{theorem}

\begin{proof}
We use induction on $n$ to prove the result. 

The base case $n=2$ is immediate from the hypotheses. So let $n>2$ and assume that the result is true for any $k<n$.

Let $p:=\wid(L_1+L_2+\hdots+L_n)$. Then $p\leq n$. Moreover the irreducible affix lemma gives that $p>1$. 

Suppose $p<n$. Then $L_1+L_2+\hdots+L_n=L'_1+L'_2+\hdots+L'_p$ for some irreducible $L'_j\in\LOfp$.

\begin{claim*}
$L_1+L_2+\hdots+L_{n'}\neq L'_1+L'_2+\hdots+L'_{p'}$ for any $n'<n$ and $p'<p$.
\end{claim*}
\begin{proof}
If $L_1+L_2+\hdots+L_{n'}= L'_1+L'_2+\hdots+L'_{p'}$ for some $n'<n$ and $p'<p$ then $L_{n'+1}+L_{n'+2}+\hdots+L_n= L'_{p'+1}+L'_{p'+2}+\hdots+L'_p$. Hence the induction hypothesis gives that $n'=p'$ and $n-n'=p-p'$, a contradiction to $p<n$.
\end{proof}

The rest of the proof can be divided into the following two cases.

\noindent{\textbf{Case I:}} $L_1$ is an $\omega$-sum.
\begin{enumerate}
    \item [(a)] $L_1$ is a proper prefix of $L'_1$: The irreducible affix lemma gives that $L'_1$ is also an $\omega$-sum. The same lemma also gives that $L_i$ is an  $\omega$-sum, where $i$ is minimum such that $L'_1$ is a prefix of $L_1+\hdots+L_i$. Using an argument similar to the proof of Corollary \ref{patchinglemma} we get that $\wid(L_1+\hdots+L_i)=1$. The induction hypothesis gives that $i=n$ which is a contradiction to $p>1$.
    \item [(b)] $L'_1$ is a proper prefix of $L_1$: Let $j$ be the maximum such that $L'_1+\hdots+L'_j$ is a prefix of $L_1$. Then $L_1$ is a prefix of $L'_1+\hdots+L'_{j+1}$. Thus a suffix of $L_1$ is a prefix of $L'_{j+1}$ which gives that $L_1,L'_{j+1}$ are both $\omega$-sums by the irreducible affix lemma. Furthermore if $i$ is the smallest such that $L'_1+\hdots+L'_{j+1}$ is a prefix of $L_1+\hdots+L_i$ then a suffix of $L'_{j+1}$ is a prefix of $L_i$ which gives that $L_i$ is an $\omega$-sum. Using an argument similar to the proof of Corollary \ref{patchinglemma} we get $\wid(L_1+\hdots+L_i)=1$, a contradiction to $i>1$.
\end{enumerate}

\noindent{\textbf{Case II:}} $L_1$ is an $\omega^*$-sum.

Let $j\geq0$ be the largest such that $L'_1+\hdots+L'_j$ is a proper prefix of $L_1$. Then $L_1$ is a proper prefix of $L'_1+\hdots+L'_{j+1}$.
\begin{enumerate}
    \item [(a)] $j=0$, $L'_1$ is an $\omega$-sum: Let $i$ be the smallest such that $L'_1$ is a prefix of $L_1+\hdots+L_i$. Since $L'_1$ is an $\omega$-sum the proof of Case I(a) goes through to obtain a contradiction.
    \item [(b)] $j>0$, $L'_{j+1}$ is an $\omega$-sum: Let $L'_{j+1}=\tilde L'_{j+1}+\bar L'_{j+1}$, where $\tilde L'_{j+1}$ is a suffix of $L_1$. By the irreducible affix lemma $\bar L'_{j+1}$ is irreducible and an $\omega$-sum. Then $$L_2+\hdots+L_n=\bar L'_{j+1}+L'_{j+2}+\hdots+L'_p.$$ By the induction hypothesis the width of the LHS is $n-1$ but the RHS has at most $p-j\leq p-1<n-1$ irreducible summands, a contradiction.
    \item [(c)] $j=0$, $L'_1$ is an $\omega^*$-sum: Let $i$ be the smallest such that $L'_1$ is a prefix of $L_1+\hdots+L_i$.
    
    If $i>2$ then $L_1+L_2$ is a prefix of $L'_1$, and hence by the dual of Lemma \ref{affixirred} we get $\wid(L_1+L_2)=1$, a contradiction. Hence $i=2$. 
    
    If $L_2$ is an $\omega^*$-sum then the second paragraph of the proof of the dual of Corollary \ref{patchinglemma} gives $\wid(L_1+L_2)=1$, which is also a contradiction. Hence $L_2$ is an $\omega$-sum. 
    
    Let $L_2=\tilde L_2+\bar L_2$ be the partition such that $L'_1=L_1+\tilde L_2$. By Lemma \ref{affixirred}, $\bar L_2$ is an $\omega$-sum. Hence by the same lemma if $\wid(\bar L_2+L_3)=1$ then $L_3$ is an $\omega$-sum. In that case Corollary \ref{patchinglemma} gives $\wid(L_2+L_3)=1$, a contradiction. Thus $\wid(\bar L_2+L_3)=2$. Hence by the induction hypothesis we get $\wid(\bar L_2+L_3+\hdots+L_n)=n-1$. But $\bar L_2+L_3+\hdots+L_n=L'_2+\hdots+L'_p$, where the RHS has fewer than $n-1$ irreducible summands, a contradiction to the above statement.
    \item[(d)] $j>0$, $L'_{j+1}$ is an $\omega^*$-sum: An easy argument using the irreducible affix lemma gives that $j=1$. By the same lemma we also get that $L'_1$ is an $\omega^*$-sum. Then an argument similar to the proof of Corollary \ref{patchinglemma} gives that $\wid(L'_1+L'_2)=1$, a contradiction in view of Remark \ref{widthofpairs} applied to the order $L'_1+\hdots+L'_p$.
\end{enumerate}
\end{proof}

When restricted to linear orders of the same rank, the width of the sum increases but the growth could be really slow.
\begin{lemma}\label{**}
Let $L,L'\in\LOfp$. If $\rk{L}=\rk{L'}$ then $$\wid(L+L')>\min\{\wid(L),\wid(L')\}.$$
\end{lemma}

\begin{proof}
Let $k:=\wid(L)$ and $m:=\wid(L')$. Without loss we may assume that $\rk{L}=\rk{L'}>0$ and $k,m\geq 1$.

Suppose $L=L_1+\hdots+L_k$ and $L'=L'_1+\hdots+L'_m$, where each $L_i$ and $L'_j$ is irreducible. Then Remark \ref{widthofpairs} gives $\wid(L_i+L_{i+1})=\wid(L'_j+L'_{j+1})=2$ for $1\leq i<k$ and $1\leq j<m$.

If $\wid(L_k+L'_1)=2$ then Theorem \ref{totalwidthfrompairs} gives $\wid(L+L')=k+m>\min\{k,m\}$. Moreover if $1\in\{k,m\}$ then the irreducible affix lemma gives that $\wid(L+L')>1$. Hence it remains to consider the case when $\wid(L_k+L'_1)=1$ and $k,m>1$. Without loss we may assume that $L_k+L'_1$ is an $\omega$-sum; the proof of the other case is dual.

Since $L_k+L'_1$ is an $\omega$-sum then $\rk{L_k}<\rk{L'_1}$ and $L'_1$ is also an $\omega$-sum by the irreducible affix lemma.

If $\wid((L_k+L'_1)+L'_2)=1$ then again by the irreducible affix lemma we conclude that $(L_k+L'_1)+L'_2$ is an $\omega$-sum and $\wid(L'_1+L'_2)=1$, a contradiction. Therefore $\wid((L_k+L'_1)+L'_2)=2$.

If $\wid(L_{k-1}+(L_k+L'_1))=1$ and $L_{k-1}+(L_k+L'_1)$ is an $\omega^*$-sum then $\wid(L_{k-1}+L_k)=1$ by the irreducible affix lemma, a contradiction. Hence if $\wid(L_{k-1}+(L_k+L'_1))=1$ then $L_{k-1}+(L_k+L'_1)$ is an $\omega$-sum.

Therefore under the hypothesis that $L_k+L'_1$ is an $\omega$-sum the above argument can be repeated to show that for $1\leq p\leq k$ if $\wid(L_p+L_{p+1}+\hdots+L_k+L'_1)=1$ then $\wid(L_p+L_{p+1}+\hdots+L_k+L'_1+L'_2)=2$ and $L_p+L_{p+1}+\hdots+L_k+L'_1$ is an $\omega$-sum.

Recall that if $L_p+L_{p+1}+\hdots+L_k+L'_1$ is an $\omega$-sum then $\rk{L_i}<\rk{L'_1}$ for $p\leq i\leq k$. Since $\rk{L}=\rk{L'}$ there is $1\leq i_0\leq k$ such that $\rk{L_{i_0}}\geq\rk{L'_1}$, and the condition $\wid(L_p+L_{p+1}+\hdots+L_k+L'_1)=1$ fails for some $p\geq i_0$. Thus Theorem \ref{totalwidthfrompairs} gives that $\wid(L+L')>\wid(L')\geq\min\{\wid(L),\wid(L')\}$.
\end{proof}

The next result generalizes Euclidean division lemma for integers.
\begin{lemma}\emph{(Euclidean division lemma)}\label{euclid}
Let $L,L'\in\LOfp$, $f:\omega\times L\to\omega\times L'$ an isomorphism such that $f(\mathbf{1}\times L)$ is a prefix of $\mathbf{1}\times L'$. Then there is $k\geq1$ and $L_1,L_2\in\LOfp$ such that $$L=L_1+L_2,\ L'=f(\mathbf{k}\times L+L_1),\ \omega\times(L_1+L_2)\cong\omega\times(L_2+L_1).$$ Moreover
\begin{itemize}
    \item if $(\rk{L_1}<\rk{L_2})$ then $\omega\times L_1$ is isomorphic to a prefix of $L_2$.
    \item if $(\rk{L_1}>\rk{L_2})$ then $\omega\times L_2$ is isomorphic to a prefix of $L_1$.
    \item if $(\rk{L_1}=\rk{L_2})$ then $\omega\times L_1\cong\omega\times L_2$ and $$\min\{\wid(L_1),\wid(L_2)\}<\min\{\wid(L),\wid(L')\}.$$
\end{itemize}
\end{lemma}

\begin{proof}
Since $f(\mathbf{1}\times L)$ is a prefix of $\mathbf{1}\times L'$ there is a largest $k\geq 1$ such that $f(\mathbf{k}\times L)$ is a prefix of $L'$. If $L'=f(\mathbf{k}\times L)+\bar L_1$ then define $L_1:=f^{-1}(\bar L_1)$. Using maximality of $k$ we obtain that $L_1$ is a prefix of (the $(k+1)^{th}$ copy of) $L$. Let $L=L_1+L_2$. Then
\begin{align*}
    \mathbf{k}\times(L_1+L_2)+L_1+\omega\times(L_2+L_1)
    &\cong\omega\times(L_1+L_2)\\
    &\cong\omega\times(\mathbf{k}\times(L_1+L_2)+L_1)\\
    &\cong \mathbf{k}\times(L_1+L_2)+L_1+\omega\times(\mathbf{k}\times(L_1+L_2)+L_1)\\
    &\cong \mathbf{k}\times(L_1+L_2)+L_1+\omega\times(L_1+L_2).
\end{align*}
Since each isomorphism above preserves the first copy of $L'\cong \mathbf{k}\times (L_1+L_2)+L_1$, we can cancel it to obtain
\begin{equation}\label{magic}
    \omega\times(L_1+L_2)\cong\omega\times(L_2+L_1)
\end{equation}

Using $\omega\times(L_2+L_1)\cong L_2+\omega\times(L_1+L_2)$ repeatedly with Equation \eqref{magic} we get, for each $m\geq 1$,
$$\omega\times(L_1+L_2)\cong \mathbf{m}\times L_2+\omega\times(L_1+L_2).$$
Thus $\omega\times L_2$ is isomorphic to a prefix of $\omega\times(L_1+L_2)$. If $\rk{L_2}<\rk{L_1}$ then since $\omega\times L_2$ is a prefix of $\omega\times(L_1+L_2)$, it is a prefix of $\mathbf{p}\times(L_1+L_2)$ for some $p\geq1$. However if $L_1$ is a proper prefix of $\omega\times L_2$ then it is in fact a prefix of $\mathbf{q}\times L_2$ for some $q\geq1$, a contradiction to $\rk{L_2}<\rk{L_1}$. Hence $\omega\times L_2$ is a prefix of $L_1$.

If $\rk{L_1}<\rk{L_2}$ then swapping $L_1$ and $L_2$ in view of Equation \eqref{magic} in the above paragraph we can obtain that $\omega\times L_1$ is isomorphic to a prefix of $L_2$.

If $\rk{L_1}=\rk{L_2}$ then Proposition \ref{****} yields isomorphisms $$\omega\times L_1\cong\omega\times(L_1+L_2)\cong\omega\times L_2.$$
The final conclusion follows from Lemma \ref{**}.
\end{proof}

\section{Euclidean algorithm for $\LOfp$}\label{eallofp}
The main goal of this section is to establish Corollary \ref{main} which is the converse of Proposition \ref{LequivgivesLINiso}. We need some more tools for that.
\begin{proposition}\label{mismatch}
Suppose $(T,s_T),(T',s_{T'})\in3\ST_\omega$ and $\bar T:=\bigcurlyvee_{i=1}^n T_i, \bar T':=\bigcurlyvee_{j=1}^m T'_j$, where $\{T_i\mid 1\leq i\leq n\}=\{T'_j\mid 1\leq j\leq m\}=\{T,T'\}$. If one of the following sets of conditions holds:
\begin{itemize}
    \item[(I)] $T\curlyvee T'\approx_L T=T_1=T'_1$;
    \item[(II)] $T\curlyvee T'\approx_L T'$,
\end{itemize}
then $\bar T^+\approx_L(\bar T')^+$.
\end{proposition}

\begin{proof}
If (I) holds and $\bar T:=\bigcurlyvee_{i=1}^n T_i$, where $\{T_i\mid 1\leq i\leq n\}=\{T,T'\}$, then it is enough to show that $\bar T^+\approx_L T^+$. Let $I:=\{i\mid1\leq i\leq n, T_i=T\}$. Since $1\in I$ we have $\overline{T}^+\approx_L(\bigcurlyvee_{i\in I} T)^+\approx_L T^+$.

If (II) holds and $\bar T:=\bigcurlyvee_{i=1}^n T_i$, where $\{T_i\mid 1\leq i\leq n\}=\{T,T'\}$, then it is enough to show that $\bar T^+\approx_L(T')^+$. Let $I:=\{i\mid1\leq i\leq n, T_i=T'\}, i_0:=\max I$ and $I':=\{i\mid i_0<i\leq n\}$. Then $\emptyset\neq I\subsetneq\{1,2,\hdots,n\}$ and
\begin{align*}
    \overline{T}^+&\approx_L(\bigcurlyvee_{i\in I} T'\curlyvee\bigcurlyvee_{i\in I'}T)^+\\&\approx_L\bigcurlyvee_{i\in I} T'\curlyvee(\bigcurlyvee_{i\in I'}T\curlyvee\bigcurlyvee_{i\in I} T')^+\\&\approx_L\bigcurlyvee_{i\in I} T'\curlyvee(\bigcurlyvee_{i\in I} T')^+\\&\approx_L T'^+.
\end{align*}
\end{proof}

There is yet another supporting result that is a necessary tool in the proof of the main theorem.
\begin{lemma} \label{***}
Suppose $(T,s_T)\in3\ST_\omega,w(T;\emptyset)=1$ and $s_T(0)=+$. If $(T,s_T)\approx_L(T_1,s_{T_1})\curlyvee(T_2,s_{T_2})$, where $\rk{\LIN{T_1}{s_{T_1}}}<\rk{\LIN{T}{s_T}}$, then there is $(T',s_{T'})\in3\ST_\omega$ such that $(T',s_{T'})\approx_L(T,s_{T})$, $(T_1,s_{T_1})$ is a prefix of $(T',s_{T'})$, and $w(T';\emptyset)=w(T_1;\emptyset)+1$.
\end{lemma}

\begin{proof}
Suppose $T=:T(0),T(1),\hdots,T(N):=(T_1,s_{T_1})\curlyvee(T_2,s_{T_2})$ is a sequence of 3STs such that, for $0\leq i<N$, $T(i)\approx_L T(i+1)$ using a basic L-equivalence or its inverse.

Let $w_i:=w(T(i);\emptyset)$. Without loss we may assume that $\max\{i\mid w_i=1\}=0$, for otherwise we may reset $0$ at such a maximum.

We inductively construct another sequence $T'(0),T'(1),\hdots,T'(N)$ such that for each $0\leq i\leq N$ we have $T'(i)\approx_L T(i)$. Let $w'_i:=w(T'(i);\emptyset)$.

Since $w_0=1$ and $s_{T(0)}(w_0-1)=+$ first observe that since $T(i)\approx_L T(0)$ we have $\hgt{T(i)_{w_i-1}}=\hgt{T(0)_0}>\hgt{T(i)_x}$ for all $1\leq i\leq N$ and $0\leq x<w_i-1$. Moreover  $s_{T(i)}(w_i-1)=+$. Similarly we can conclude that $s_{T'(i)}(w'_i-1)=+$ for $0\leq i\leq N$.

For each $i>0$ we write $\tilde T(i):=T(i)_0\curlyvee T(i)_1\curlyvee\hdots\curlyvee T(i)_{w_i-2}$ so that $T(i)=\tilde T(i)\curlyvee T(i)_{w_i-1}$. We decompose $T'(i)$ similarly.

We inductively construct $T'(i+1)$ using $T'(i)$ while ensuring the following inductive hypothesis.

\noindent{}(IH) For each $0\leq i\leq N$, there are $k_i>m_i\geq0$ and $\bar T(i)\in3\ST_\omega$ satisfying $$T'(i)=\tilde T(i)\curlyvee(\bigcurlyvee_{j=1}^{m_i} \widehat {T(i)}_{w_i-1})\curlyvee\Big(\widehat {T(i)}_{w_i-1}\curlyvee\bar T(i)\curlyvee\tilde T(i)\curlyvee\bigcurlyvee_{j=1}^{m_i} \widehat {T(i)}_{w_i-1}\Big)^+,$$ 
\begin{equation}\label{leq}
    \bigcurlyvee_{s=1}^{k_i}\widehat {T(i)}_{w_i-1}\approx_L\bar T(i)\curlyvee\tilde T(i)\curlyvee(\bigcurlyvee_{j=1}^{m_i} \widehat {T(i)}_{w_i-1})
\end{equation}

We clearly have
\begin{alignat*}{3}
    T(i)&=\tilde T(i)\curlyvee\widehat {T(i)}_{w_i-1}^+&\\
    &\approx_L\tilde T(i)\curlyvee\Big(\bigcurlyvee_{s=1}^{k_i+1}\widehat {T(i)}_{w_i-1}\Big)^+&[(k_i+1)\mbox{-}\mathsf{REPL}]\\
    &\approx_L\tilde T(i)\curlyvee(\bigcurlyvee_{j=1}^{m_i} \widehat {T(i)}_{w_i-1})\curlyvee\Big(\bigcurlyvee_{s=1}^{k_i+1}\widehat {T(i)}_{w_i-1}\Big)^+&[\mbox{iterated }\mathsf{EXUDE}]\\
    &\approx_L\tilde T(i)\curlyvee(\bigcurlyvee_{j=1}^{m_i} \widehat {T(i)}_{w_i-1})\curlyvee\Big(\widehat {T(i)}_{w_i-1}\curlyvee\bar T(i)\curlyvee\tilde T(i)\curlyvee\bigcurlyvee_{j=1}^{m_i} \widehat {T(i)}_{w_i-1}\Big)^+&[\mbox{Equation } (\ref{leq})]\\
    &=T'(i)&
\end{alignat*}

Suppose $\tilde T(N)=T_1\curlyvee T_3$ for some $T_3\in3\ST_\omega$. Then the 3ST required by the statement of the lemma can be chosen to be $$T':=T_1\curlyvee\Big(T_3\curlyvee\bigcurlyvee_{j=1}^{m_N+1}\widehat {T(N)}_{w_N-1}\curlyvee\bar T(N)\curlyvee T_1\Big)^+$$ and it is clear that $T'\approx_L T'(N)$.

For the base case choose $T'(0):=T(0)$ so that IH readily holds.

For the inductive case assume that for some $i<N$, $T'(i)$ has been constructed and we construct $T'(i+1)$ in various cases as follows.

\noindent{\textbf{Case I:}} Suppose $T(i+1)_{w_{i+1}-1}=T(i)_{w_i-1}$. Then $\tilde T(i)\approx_L\tilde T(i+1)$. Using this L-equivalence repeatedly and choosing $k_{i+1}:=k_i$, $m_{i+1}:=m_i$, and $\bar T(i+1):=\bar T(i)$ we can readily verify that $T'(i+1)\approx_L T'(i)$.

\noindent{\textbf{Case II:}} Suppose $\tilde T(i)=\tilde T(i+1)$ and $\widehat{T(i+1)}_{w_{i+1}-1}\approx_L \widehat{T(i)}_{w_i-1}$. Using this L-equivalence repeatedly and choosing $k_{i+1}:=k_i$, $m_{i+1}:=m_i$, and $\bar T(i+1):=\bar T(i)$ we can readily verify that $T'(i+1)\approx_L T'(i)$.

\noindent{\textbf{Case III:}} Suppose $T(i+1)=\mathsf{EXUDE}((T(i),s_{T(i)});w_i-1)$. Then there are 3STs $T_a,T_b$ with $w(T_a;\emptyset)=1$ such that 
$$T(i) = \tilde{T}(i) \curlyvee \Big(T_a\curlyvee T_b\Big)^+,\ T(i+1) = \tilde{T}(i) \curlyvee T_a \curlyvee \Big(T_b \curlyvee T_a\Big)^+.$$
For brevity, let $T_c:=\tilde T(i)$. Then we get a sequence of L-equivalences where the reasons for each step are written in square brackets at the end of the line.
\begin{alignat*}{3}
    T'(i) &= T_c\curlyvee\bigcurlyvee_{j=1}^{m_i} (T_a \curlyvee T_b)\curlyvee\Big((T_a \curlyvee T_b)\curlyvee\bar T(i)\curlyvee T_c\curlyvee\bigcurlyvee_{j=1}^{m_i} ((T_a \curlyvee T_b))\Big)^+ &[IH]\\
    &\approx_LT_c\curlyvee\bigcurlyvee_{j=1}^{m_i} (T_a \curlyvee T_b)\curlyvee\Big(\bigcurlyvee_{s=1}^{k_i+1} (T_a \curlyvee T_b)\Big)^+ &[\mbox{Equation } (\ref{leq})]\\
    &\approx_LT_c\curlyvee\bigcurlyvee_{j=1}^{m_i} (T_a \curlyvee T_b)\curlyvee\Big(\bigcurlyvee_{s=1}^{2(k_i+1)} (T_a \curlyvee T_b)\Big)^+ &[2\mbox{-}\mathsf{REPL}] \\
    &\approx_LT_c\curlyvee\bigcurlyvee_{j=1}^{m_i} (T_a \curlyvee T_b)\curlyvee T_a\curlyvee\Big(\bigcurlyvee_{s=1}^{2(k_i+1)} (T_b\curlyvee T_a)\Big)^+ &[\mathsf{EXUDE}] \\
    &\approx_LT_c\curlyvee T_a\curlyvee\bigcurlyvee_{j=1}^{m_i} (T_b \curlyvee T_a) \curlyvee\Big((T_b\curlyvee T_a)\curlyvee\bar T(i+1)\curlyvee T_c\curlyvee T_a\curlyvee\bigcurlyvee_{j=1}^{m_i} (T_b \curlyvee T_a)\Big)^+ &[\mbox{Equation } (\ref{leq})],
\end{alignat*}
where $\bar T(i+1):=\bigcurlyvee_{s=1}^{k_i}(T_b\curlyvee T_a)\curlyvee T_b\curlyvee \bar{T}(i)$, $m_{i+1}:=m_i$ and $k_{i+1}:=2k_i+1$.

\noindent{\textbf{Case IV:}} Suppose $T(i)=\mathsf{EXUDE}((T(i+1),s_{T(i+1)});w_{i+1}-1)$. Then there are 3STs $T_a,T_b,T_c$ with $w(T_b;\emptyset)=1$ such that $$T(i+1)=T_a\curlyvee(T_b\curlyvee T_c)^+,\ T(i)=T_a\curlyvee T_b\curlyvee\Big(T_c\curlyvee T_b\Big)^+.$$
Then
\begin{alignat*}{3}
    T'(i)&=T_a\curlyvee T_b\curlyvee\bigcurlyvee_{j=1}^{m_i} (T_c\curlyvee T_b)\curlyvee\Big((T_c\curlyvee T_b)\curlyvee\bar T(i)\curlyvee T_a\curlyvee T_b\curlyvee(\bigcurlyvee_{j=1}^{m_i} (T_c\curlyvee T_b))\Big)^+&[IH]\\
    &\approx_L T_a\curlyvee T_b\curlyvee\bigcurlyvee_{j=1}^{m_i} (T_c\curlyvee T_b)\curlyvee\Big(\bigcurlyvee_{s=1}^{k_i+1} (T_c\curlyvee T_b)\Big)^+&[\mbox{Equation } (\ref{leq})]\\
    &\approx_L T_a\curlyvee T_b\curlyvee\bigcurlyvee_{j=1}^{m_i} (T_c\curlyvee T_b)\curlyvee\Big(\bigcurlyvee_{s=1}^{2(k_i+1)} (T_c\curlyvee T_b)\Big)^+&[2\mbox{-}\mathsf{REPL}]\\
    &\approx_L T_a\curlyvee \bigcurlyvee_{j=1}^{m_i+1} (T_b\curlyvee T_c)\curlyvee\Big(\bigcurlyvee_{s=1}^{2(k_i+1)} (T_b\curlyvee T_c)\Big)^+&[\mbox{iterated }\mathsf{EXUDE}]\\
    &\approx_L T_a\curlyvee \bigcurlyvee_{j=1}^{m_i+1} (T_b\curlyvee T_c)\curlyvee\Big((T_b\curlyvee T_c)\curlyvee\bar T(i+1)\curlyvee T_a\curlyvee\bigcurlyvee_{j=1}^{m_i+1} (T_b\curlyvee T_c)\Big)^+&[\mbox{Equation } (\ref{leq})],\\
\end{alignat*}
where $\bar T(i+1):=\bigcurlyvee_{s=1}^{k_i} (T_b\curlyvee T_c)\curlyvee T_b\curlyvee\bar T(i)$, $m_{i+1}:=m_i+1$ and $k_{i+1}:=2k_i+1$.

\noindent{\textbf{Case V:}} Suppose $T(i+1)=n\mbox{-}\mathsf{REPL}((T(i),s_{T(i)});w_i-1)$ for some $n>1$. For brevity, let $T_a:=\tilde T(i),T_b:=\widehat {T(i)}_{w_i-1}$. Then $$T(i)=T_a\curlyvee T_b^+,\ T(i+1)=T_a\curlyvee\Big(\bigcurlyvee_{t=1}^n T_b\Big)^+.$$
Then
\begin{alignat*}{3}
    T'(i)&=T_a\curlyvee\bigcurlyvee_{j=1}^{m_i} T_b\curlyvee\Big(T_b\curlyvee\bar T(i)\curlyvee T_a\curlyvee\bigcurlyvee_{j=1}^{m_i}T_b\Big)^+&[IH]\\
    &\approx_LT_a\curlyvee\bigcurlyvee_{j=1}^{m_i}T_b\curlyvee\Big(\bigcurlyvee_{s=1}^{k_i+1}T_b\Big)^+&[\mbox{Equation } (\ref{leq})]\\
    &\approx_LT_a\curlyvee\bigcurlyvee_{j=1}^{m_i}T_b\curlyvee\Big(\bigcurlyvee_{s=1}^{pn(k_i+1)}T_b\Big)^+&[pn\mbox{-}\mathsf{REPL}]\\
    &\approx_LT_a\curlyvee\bigcurlyvee_{j=1}^{m_in}T_b\curlyvee\Big(\bigcurlyvee_{s=1}^{pn(k_i+1)}T_b\Big)^+&[\mbox{iterated }\mathsf{EXUDE}]\\
    &\approx_LT_a\curlyvee\bigcurlyvee_{j=1}^{m_in}T_b\curlyvee\Big(\bigcurlyvee_{t=1}^{n}T_b\curlyvee\bigcurlyvee_{s=1}^{pn(k_i+1)-m_i(n-1)-n-k_i} T_b\curlyvee\bar T(i)\curlyvee T_a\curlyvee\bigcurlyvee_{j=1}^{m_in}T_b\Big)^+&[\mbox{Equation } (\ref{leq})]\\
\end{alignat*}
The integer $p$ could be chosen so that $pn(k_i+1)-m_i(n-1)-n-k_i>0$. Hence choosing $m_{i+1}:=m_i n$, $k_{i+1}:=p(k_i+1)-1$, $\bar T(i+1):=\bigcurlyvee_{s=1}^{pn(k_i+1)-m_i(n-1)-n-k_i} T_b\curlyvee\bar T(i)$ does the job.

\noindent{\textbf{Case VI:}} Suppose $T(i) = n\mbox{-}\mathsf{REPL}((T(i+1),s_{T(i+1)});w_{i+1}-1)$ for some $n>1$. For brevity, let us take $T_a := \tilde T(i+1)$, $T_b := \widehat{T(i+1)}_{w_{i+1}-1}$. Then $$T(i) = T_a\curlyvee\Big(\bigcurlyvee_{t=1}^n T_b\Big)^+, T(i+1) = T_a\curlyvee T_b^+.$$ Then by the induction hypothesis we have
$$T'(i) = T_a\curlyvee\bigcurlyvee_{j=1}^{nm_i} T_b\curlyvee\Big(\bigcurlyvee_{t=1}^{n}T_b\curlyvee\bar T(i)\curlyvee T_a\curlyvee\bigcurlyvee_{j=1}^{nm_i}T_b\Big)^+.$$
In this case we choose $m_{i+1}:= nm_i$, $k_{i+1}:=n(k_i+1)-1$, and $\bar{T}(i+1):= \bigcurlyvee_{t=1}^{n-1}T_b \curlyvee\bar{T}(i)$.
\end{proof}

We first prove a more flexible version of the main goal.
\begin{theorem}\label{affixembedthm}
Suppose $(T,s_T),(T',s_{T'})\in3\ST_\omega$ and $f:\LIN{T}{s_T}\to\LIN{T'}{s_{T'}}$ is a prefix (resp. suffix) embedding. Then there are $(\tilde T,s_{\tilde T}),(\tilde T',s_{\tilde T'})\in3\ST_\omega$ such that $(T,s_T)\approx_L(\tilde T,s_{\tilde T})$, $(T',s_{T'})\approx_L(\tilde T',s_{\tilde T'})$, and $(\tilde T,s_{\tilde T})$ is a prefix (resp. suffix) of $(\tilde T',s_{\tilde T'})$ such that if $g:\LIN{T}{s_T}\to \LIN{\tilde T}{s_{\tilde T}}$ and $h:\LIN{\tilde T'}{s_{\tilde T'}}\to \LIN{T'}{s_{T'}}$ are the induced isomorphims (in the sense of Proposition \ref{LequivgivesLINiso}) then $h\mid_{\LIN{\tilde T}{s_{\tilde T}}}=fg^{-1}$.
\end{theorem}

\begin{proof}
Let $L:=\LIN{T}{s_T}, L':=\LIN{T'}{s_{T'}}$, $n:=\rk L, n':=\rk{L'}, w:=w(T;\emptyset)$ and $w':=w(T';\emptyset)$. Without loss we may assume that $f:L\to L'$ is a prefix embedding. Clearly $n\leq n'$.

We shall use transfinite induction on the pair $(n,n')$ to prove the result, where the set $\{(n,n')\in\N\times\N\mid n\leq n'\}$ of such pairs of ranks is arranged in the lexicographic ordering, denoted $\lex$. Also note that we implicitly keep track of all the isomorphisms while invoking the induction hypothesis. 

For the base case, we have $(n,n')=(0,0)$, and the conclusion is obvious irrespective of the values of $w,w'$.

\noindent{\textbf{Case I:}} $n=n'>0,\ w=w'=1$.

If $L'$ is an $\omega$-sum then $f(L)=L'$ for if $f(L)$ is a proper prefix of $L'$ then $n<n'$ by the irreducible affix lemma. Thus $L$ is also an $\omega$-sum. On the other hand if $L'$ is an $\omega^*$-sum then the irreducible affix lemma gives that $L$ is also an $\omega^*$-sum.

\begin{enumerate}
    \item [(a)] $f$ is not surjective: The above discussion gives that $L'$ is an $\omega^*$-sum. If $L'=f(L)+L''$ then Lemma \ref{**} gives that $\rk {L''}<n$. Propositions \ref{3STLOfpcorr} and \ref{summandfp} together yield $ (T'',s_{T''}) \in 3\ST_\omega$ such that $\LIN{T''}{s_{T''}}\cong L''$. 
    
    Applying the induction hypothesis to the suffix embedding $L''\to L'$ we get 3STs $\bar T',\bar T''$ such that $\bar T' \approx_L T'$, $\bar T'' \approx_L T''$ and $\bar T'=\tilde T'\curlyvee\bar T''$ for some 3ST $\tilde T'$. Then $\LIN{T}{s_T} \cong \LIN{T'}{s_{T'}}$ and we are in the next subcase.
    \item [(b)] $f$ is a bijection: Without loss we assume that both $L,L'$ are $\omega$-sums; the other argument is dual. 
    
    Let $L\cong\omega\times\tilde L$ and $L'\cong\omega\times\tilde L'$. Set $(T_{0,1},s_{T_{0,1}}):=(\widehat T_0,s_{\widehat T_0})$, $(T_{0,2},s_{T_{0,2}}):=(\widehat T'_0,s_{\widehat T'_0})$ and $\tilde L_{0,j}:=\LIN{T_{0,j}}{s_{T_{0,j}}}$ for $j=1,2$ so that $L_{0,1}=\tilde L$ and $L_{0,2}=\tilde L'$. Then repeatedly applying the Euclidean division lemma (Lemma \ref{euclid}) starting with $L_{0,1}$ and $L_{0,2}$ gives a positive integer $m$ and for each $1\leq i\leq m$ a permutation $\xi_i$ of $\{1,2\}$, a positive integer $k_i$, and $L_{i,1},L_{i,2}\in\LOfp$ such that
    \begin{itemize}
    \item $L_{i-1,\xi_i(1)}=L_{i,1}+L_{i,2}$ and $L_{i-1,\xi_i(2)}=f(\mathbf{k_i}\times(L_{i,1}+L_{i,2})+L_{i,1})$;
    \item $\rk{L_{i,1}}=\rk{L_{i,2}}$ for each $0\leq i<m$;
    \item $\rk{L_{m,1}}\neq\rk{L_{m,2}}$.
    \end{itemize}
    For each $1\leq i\leq m,1\leq j\leq 2$ since $\rk{L_{i,j}}<n = n'$, the induction hypothesis together with Propositions \ref{3STLOfpcorr} and \ref{summandfp} gives $(T_{i,j},s_{T_{i,j}})\in3\ST_\omega$ such that
\begin{itemize}
    \item $\LIN{T_{i,j}}{s_{T_{i,j}}}\cong L_{i,j}$;
    \item $T_{i-1,\xi_i(1)}\approx_L T_{i,1}\curlyvee T_{i,2}$ and $T_{i-1,\xi_i(2)}\approx_L\bigcurlyvee_{l=1}^{k_i}(T_{i,1}\curlyvee T_{i,2})\curlyvee T_{i,1}$.
\end{itemize}

If $\rk{L_{m,2}}<\rk{L_{m,1}}$ then the Euclidean division lemma gives that $\omega\times L_{m,2}$ is isomorphic to a prefix of $L_{m,1}$. Using the induction hypothesis for this prefix embedding we obtain $T'_{m,j}\approx_L T_{m,j}$ for $j=1,2$ such that $T'_{m,1}= (T'_{m,2})^+\curlyvee\bar T$ for some 3ST $\bar T$ so that $T'_{m,2}\curlyvee T'_{m,1}\approx_L T'_{m,1}$.

On the other hand if $\rk{L_{m,1}}<\rk{L_{m,2}}$ then the Euclidean division lemma gives that $\omega\times L_{m,2}$ is isomorphic to a prefix of $L_{m,1}$. Using the induction hypothesis for this prefix embedding we obtain $T'_{m,j}\approx_L T_{m,j}$ for $j=1,2$ such that $T'_{m,2}=(T'_{m,1})^+\curlyvee\bar T$ for some 3ST $\bar T$ so that $T'_{m,1}\curlyvee T'_{m,2}\approx_L T'_{m,2}$.

In either of the above two cases $T_{0,1},T_{0,2}$ are L-equivalent to finite joins of $T'_{m,1}$ and $T'_{m,2}$ so that the hypotheses of Proposition \ref{mismatch} are satisfied, which then yields $T=(\widehat T_0)^+=T_{0,1}^+\approx_LT_{0,2}^+=(\widehat T'_0)^+=T'$ as required.
\end{enumerate}
\noindent{\textbf{Case II:}} $0\leq n<n'$, $w=w'=1$.

Since $n<n'$, the irreducible affix lemma gives that $L'$ is an $\omega$-sum. Thus $L$ is a prefix of $\mathbf{k}\times\tilde L'$ for some $k>0$. Since $(\rk{L},\rk{\mathbf{k}\times L})\lex(n,n')$, the induction hypothesis applied to the prefix embedding $L\to(\mathbf{k}\times \tilde L)$ gives 3STs $\bar T,\bar T'$ such that $\bar T\approx_L T$ and $\bigcurlyvee_{i=1}^k\widehat{T'}_0\approx_L\bar T\curlyvee\bar T'$. Then
$$T'\approx_L(\widehat{ T'}_0)^+\approx_L(\bigcurlyvee_{i=1}^k\widehat{T'}_0)^+\approx_L(\bar T\curlyvee\bar T')^+\approx_L\bar T\curlyvee(\bar T'\curlyvee\bar T)^+,$$
as required.

\noindent{\textbf{Case III:}} $0\leq n\leq n'$, $w=1<w'$.

Let $L'_j:=\LIN{T'_j}{s_{T'_j}}$ for $0\leq j<w'$ and $k\geq 0$ the minimum such that $f(L)$ is a prefix of $L'_0+L'_1+\hdots+L'_k$. Then $L'_0+L'_1+\hdots+L'_{k-1}$ is a prefix of $f(L)$.
\begin{enumerate}
    \item [(a)] $s_T(0)=+$ : Using the irreducible affix lemma we get that $s_{T'}(k)=+$ and $\rk{L'_0+L'_1+\hdots+L'_{k-1}}<n$.
    
    If $k>0$ then let $T^{(0)}:=T$. Applying the induction hypothesis and Lemma \ref{***} to the prefix embedding of $f^{-1}(L'_0)\to L$ yields 3STs $T^{(1)},\bar{T'}_0$ such that $\bar{T'}_0\approx_L T'_0$, $\bar{T'}_0\curlyvee T^{(1)}\approx_L T^{(0)}$ and $\wid(T^{(1)})=1$. Repeating this procedure $k$ times we get $T^{(k)}$ with $\wid(T^{(k)})=1$ so that $\LIN{T^{(k)}}{s_{T^{(k)}}}$ is isomorphic to a prefix of $L'_k$.
    
    For all $k\geq 0$, if $\rk{L'_k}=n=\rk{\LIN{T^{(k)}}{s_{T^{(k)}}}}$ then using Case I, otherwise using Case II for this prefix embedding we get the required conclusion.
    
    \item [(b)] $s_T(0)=-$ : If $k=0$ then there are two possibilities. If $\rk{L'_0}=n$ then Case I, otherwise Case II for this prefix embedding gives the required conclusion.
    
    If $k>0$ then we have $L=f^{-1}(L_0)+\widetilde{L}$ for some $\widetilde{L}\in \LOfp$. Since $s_T(0)=-$, i.e.,. $L$ is an $\omega^*$-sum, we can apply the dual version of Lemma \ref{affixirred} to get $\rk{\widetilde{L}}<\rk{L'_0} = \rk{L} = n$. This also means that the argument of Case I can be applied to the embedding of $L'_0$ into $L$, so that without loss we have $T \approx_L T'_0\curlyvee \widetilde{T}$ for some 3ST $\widetilde{T}$ satisfying $\LIN{\widetilde{T}}{s_{\widetilde{T}}} \cong \widetilde{L}$. Then we are left to find two 3STs, namely $\overline T\approx_L\widetilde{T}$ and $\overline{\overline T}$, so that $\overline T\curlyvee\overline{\overline{T}}\approx_L T'_1 \curlyvee \ldots \curlyvee T'_k$, which can be achieved by the inductive hypothesis since we established above that $\rk{\widetilde{L}} < n$. Therefore we are done.
\end{enumerate}
\noindent{\textbf{Case IV:}} $0\leq n\leq n'$, $1<w$

Here the argument is completed by embedding $\LIN{T_i}{s_{T_i}}$ into the target linear order in the increasing order for each $0\leq i<w$ using Case III.

This completes the proof.
\end{proof}

Here is the promised converse to Proposition \ref{LequivgivesLINiso} that follows readily from the above theorem.
\begin{corollary}\label{main}
Suppose $(T,s_T),(T',s_{T'})\in3\ST_\omega$. If $f:\LIN{T}{s_T}\to\LIN{T'}{s_{T'}}$ is an isomorphism then $(T,s_T)\approx_L(T',s_{T'})$.
\end{corollary}

\printbibliography
\end{document}